\documentclass{amsart}
\usepackage{amsfonts, bbm}

\usepackage{graphicx}
\usepackage{psfrag}

\newtheorem{theorem}{Theorem}[section]
\theoremstyle{plain}
\newtheorem{corollary}[theorem]{Corollary}
\newtheorem{lemma}[theorem]{Lemma}
\newtheorem{proposition}[theorem]{Proposition}
\theoremstyle{remark}

\numberwithin{equation}{section}

\newcommand{\vol}{\operatorname{vol}}
\newcommand{\ovol}{\operatorname{0-vol}}
\newcommand{\FP}{\operatornamewithlimits{FP}}

\newcommand{\re}{\operatorname{Re}}
\newcommand{\im}{\operatorname{Im}}
\newcommand{\res}{\operatorname{Res}}
\newcommand{\rank}{\operatorname{rank}}

\newcommand{\supp}{\operatorname{supp}}

\newcommand{\dist}{\operatorname{dist}}

\newcommand{\norm}[1]{\Vert #1 \Vert}

\newcommand{\brak}[1]{\langle #1 \rangle}

\newcommand{\bbR}{\mathbb{R}}
\newcommand{\bbH}{\mathbb{H}}
\newcommand{\bbC}{\mathbb{C}}
\newcommand{\bbZ}{\mathbb{Z}}
\newcommand{\bbN}{\mathbb{N}}
\newcommand{\bbB}{\mathbb{B}}

\newcommand{\calR}{\mathcal{R}}

\newcommand{\cinf}{C^\infty}
\newcommand{\del}{\partial}

\newcommand{\FPe}{\FP_{\varepsilon\to 0}}

\newcommand{\vep}{\varepsilon}

\newcommand{\tN}{\widetilde{N}}

\newcommand{\chr}{\mathbbm{1}}

\newcommand{\bF}{\mathbf{F}}
\newcommand{\Ai}{{\rm Ai}}

\newcommand{\diX}{\del_\infty X}
\newcommand{\diY}{\del_\infty Y}
\newcommand{\diC}{\del_\infty C}
\newcommand{\rf}{{\rm f}}
\newcommand{\rc}{{\rm c}}
\newcommand{\rff}{{\rm ff}}
\newcommand{\rfc}{{\rm fc}}
\newcommand{\rcf}{{\rm cf}}
\newcommand{\rcc}{{\rm cc}}

\newcommand{\nf}{{n_\rf}}
\newcommand{\nc}{{n_\rc}}

\begin{document}

\title[Sharp upper bounds]{Sharp geometric upper bounds on resonances for surfaces with hyperbolic ends}
\author[Borthwick]{David Borthwick}
\address{Department of Mathematics and Computer Science, Emory
University, Atlanta, Georgia, 30322, USA}
\thanks{Supported in part by NSF\ grant DMS-0901937.}
\email{davidb@mathcs.emory.edu}
\date{\today}

\begin{abstract}
We establish a sharp geometric constant for the upper bound on the resonance counting 
function for surfaces with hyperbolic ends.  An arbitrary metric is allowed within some compact
core, and the ends may be of hyperbolic planar, funnel, or cusp type.  
The constant in the upper bound depends
only on the volume of the core and the length parameters associated to 
the funnel or hyperbolic planar ends.  Our estimate is sharp in that it
reproduces the exact asymptotic constant in the case of 
finite-area surfaces with hyperbolic cusp ends, and also in the case of funnel ends with Dirichlet boundary
condtiions.  
\end{abstract}

\maketitle
\tableofcontents

\bigbreak
\section{Introduction}\label{intro.sec}

For a compact Riemannian surface, the Weyl law shows that the asymptotic distribution
of eigenvalues is determined by global geometric quantities.  In the compact hyperbolic
case, Weyl asymptotics follow easily from the Selberg trace formula, see e.g.~\cite{McKean:1972}, 
and this approach extends also to non-compact hyperbolic surfaces of finite area \cite{Venkov}.
Some reinterpretation of the spectral counting is needed for the non-compact case;
one can either supplement the counting function for the discrete spectrum by a term related 
to the scattering phase, or else use the counting function for 
resonances instead of eigenvalues.  Weyl asymptotics, in this extended sense, 
were established for general finite-area surfaces with hyperbolic cusp ends by 
M\"uller \cite{Muller:1992} and Parnovski \cite{Parnovski:1995}. 

For infinite-area surfaces with hyperbolic ends, the discrete spectrum is finite and possibly
empty, and therefore plays no role in the spectral asymptotics.  
One could look for analogies to the finite-area results in the asymptotics of either the scattering phase
or the resonance counting function.  For the scattering phase of a surface with hyperbolic ends,
Weyl asymptotics were proven by Guillop\'e-Zworski \cite{GZ:1997}.
One does not necessarily expect a corresponding result to hold 
for the resonance counting function---see e.g.~\cite[Remark~1.6]{GZ:1997}, 
but neither can we rule out the possibility at this point.
The issue of how global geometric properties influence the distribution of resonances a 
remains a compelling problem.

At present, only the order of growth of the resonance counting function is well understood.
Guillop\'e-Zworski \cite{GZ:1995a, GZ:1997} showed
the the resonance counting function for infinite-area surfaces with hyperbolic ends 
satisfies $N_g(t) \asymp t^2$ (with the caveat that the lower bound is proportional to 
the $0$-volume which might be zero in exceptional cases).
These results have been extended to higher dimensional manifolds with hyperbolic
ends in Borthwick \cite{Borthwick:2008}.   Unfortunately, the methods used in these proofs
yield only an ineffective constant for the upper bound, with no clear geometric content.  
Moreover, the derivation of the lower bound depends explicitly on the upper bound,
so the geometric dependence of the lower bound was likewise unknown.

In this paper we present a geometric constant for the upper bound on the resonance counting function
for infinite-area surfaces with hyperbolic ends.  This constant is sharp in the sense that it 
agrees with the exact asymptotics in the cases of finite area surfaces or truncated funnnels.  
Our approach is inspired by Stefanov's recent
paper \cite{Stefanov:2006} on compactly supported perturbations of the Laplacian on $\bbR^n$ for $n$ odd,
and similar techniques were applied to compactly supported perturbations of $\bbH^{n+1}$ in Borthwick
\cite{Borthwick:2009}.  

We can state the cleanest result for a hyperbolic surface $(X,g) \cong \bbH^2/\Gamma$.   
Let $\calR_g$ denote the
associated resonance set (poles of the meromorphic continuation of $(\Delta_g - s(1-s))^{-1}$),
with counting function
$$
N_g(t) :=  \#\bigl\{\zeta \in \calR_g:\> |\zeta - \tfrac12| \le t\bigr\}.
$$
The sharp version of our bound involves a regularization of the counting function,
\begin{equation}\label{tN.def}
\tN_g(a) := \int_0^a \frac{2N_g(t)}{t^2}\>dt.
\end{equation}
This type of regularization is standard in the theory of zeros of entire functions, 
and there is a natural connection to the asymptotics of $N_g(t)$,
$$
\tN_g(a) \sim Ba^2 \quad\Longleftrightarrow\quad N_g(t) \sim Bt^2.
$$
(see \cite[Lemma~1]{Stefanov:2006}).  
If we work only with upper bounds, then we lose some sharpness in the
estimate,
$$
\tN_g(a) \le Ba^2 \quad\Longrightarrow\quad N_g(t) \le eBt^2.
$$

\begin{theorem}\label{hsurf.thm}
Suppose $(X,g)$ is a geometrically finite hyperbolic surface with $\chi(X) < 0$.  Let 
$\ell_1,\dots, \ell_\nf$ denote the diameters of the geodesic boundaries of the funnels of $X$.   
The regularized counting function for the resonances of $\Delta_g$ satisfies
\begin{equation}\label{Ng.bnd}
\frac{\tN_{g}(a)}{a^2} \le |\chi(X)| +  \sum_{j=1}^\nf \frac{\ell_j}{4}  + o(1).
\end{equation}
\end{theorem}
We can see that this result is sharp in two extreme cases.  For a finite-area hyperbolic surface
(i.e.~$\nf=0$), our upper bound agrees with the known asymptotic $N_g(t)/t^2 \sim |\chi(X)|$.    
Moreover, for an isolated hyperbolic funnel $F_\ell$
of boundary length $\ell$, under Dirichlet boundary conditions, the resonances form a half-lattice.  
It's then easy to see that $N_{F_\ell}(t)/t^2 \sim \ell/4$, so
the funnel portion of (\ref{Ng.bnd}) is also sharp.

The restriction to $\chi(X)<0$ in Theorem~\ref{hsurf.thm} leaves out just a few cases.  The 
complete (smooth) hyperbolic surfaces for which $\chi(X) \ge 0$ are the hyperbolic plane $\bbH^2$, 
the hyperbolic cylinder $C_\ell := \bbH^2/\brak{z \mapsto e^{\ell} z}$, and the parabolic cylinder 
$C_\infty := \bbH^2/\brak{z \mapsto z+1}$.   Resonance sets can be computed explicitly in these
cases, and exact asymptotics for the counting function are easily obtained:
$$
N_{\bbH^2}(t) \sim t^2, \qquad N_{C_\ell}(t) \sim \frac{\ell}2t^2, \qquad N_{C_\infty}(t) = 1.
$$
If we interpret $C_\ell$ as the union of 2 funnel ends, then (\ref{Ng.bnd}) would also give
a sharp estimate for this case.

Using Theorem~\ref{hsurf.thm} in conjunction with the Guillop\'e-Zworski argument \cite{GZ:1997}
for the lower bound, we can deduce the following:
\begin{corollary}\label{lb.cor}
For $k \in\bbN$ there exists a constant $c_k$ such that for any geometrically finite hyperbolic surface $(X,g)$ 
with $\chi(X) < 0$,
$$
\frac{N_g(t)}{t^2} \ge c_k \> |\chi(X)|\>  \Biggl( 1 + \frac{1}{|\chi(X)|} \sum_{j=1}^\nf \frac{\ell_j}{4} \Biggr)^{-\frac{2}{k}},
$$
for $t \ge 1$.
\end{corollary}
\noindent
The constant $c_k$ obtained in this way (see \S\ref{scdet.sec} for the derivation) 
is rather ineffective;  the point here is
just that we can find a lower bound that depends only on $\chi(X)$ and $\{\ell_j\}$.

We will obtain Theorem~\ref{hsurf.thm} as a consequence of a somewhat more general estimate.  
Consider a smooth Riemannian surface $(X,g)$, possibly with boundary, 
which has finitely many ends that are assumed to be of hyperbolic planar, funnel, or cusp type.
That is, $X$ admits the following decomposition, as illustrated in Figure~\ref{KFC},
\begin{equation}\label{X.KYC}
X = K \sqcup Y_1 \sqcup \dots \sqcup Y_{\nf} \sqcup C_{\nf+1} \sqcup \dots \sqcup C_{\nf+\nc},
\end{equation}
where the core $K$ is a compact manifold with boundary.
The metric in $K$ is arbitrary.
The $Y_j$'s are infinite-area ends: either hyperbolic planar,
\begin{equation}\label{hp.end}
Y_j \cong [b_j, \infty) \times S^1, \quad g|_{Y_j} = dr^2 + \sinh^2r\>  d\theta^2, \quad b_j \ge 0,
\end{equation}
or hyperbolic funnels,
\begin{equation}\label{f.end}
Y_j \cong [b_j, \infty) \times S^1, \quad g|_{Y_j} = dr^2 +  \ell_j^2 \cosh^2r \frac{d\theta^2}{(2\pi)^2}, \quad 
b_j, \ell_j \ge 0.
\end{equation}
The $C_j$'s are hyperbolic cusps,
\begin{equation}\label{c.end}
C_j \cong [b_j, \infty) \times S^1, \quad g|_{C_j} = dr^2 + e^{-2r}  \frac{d\theta^2}{(2\pi)^2}, \quad b_j \ge 0.
\end{equation}
The finite-area portion of $X$ consisting of the core plus the cusps is denoted by 
\begin{equation}\label{Xc.def}
X_\rc := K \sqcup C_{\nf+1} \sqcup \dots \sqcup C_{\nf+\nc}.
\end{equation}

\begin{figure} 
\psfrag{K}{$K$}
\psfrag{bK}{$\partial K$}
\psfrag{C1}{$C_1$}
\psfrag{F1}{$Y_1$}
\psfrag{F2}{$Y_2$}
\psfrag{Xc}{$X_\rc$}
\begin{center}  
\includegraphics{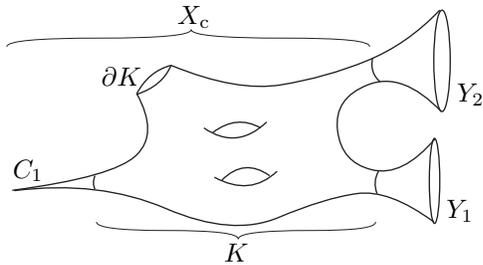} 
\end{center}
\caption{Surface $X$ with boundary and hyperbolic ends.}\label{KFC}
\end{figure}

Note that any geometrically finite hyperbolic surface, with the exception of the parabolic cylinder $C_\infty$, 
admits a decomposition of the form (\ref{X.KYC}).  And in such surfaces, aside from $\bbH^2$ itself, 
only funnel or cusp ends can occur.

We let $\Delta_g$ denote the positive Laplacian on $(X,g)$.  In general we may consider the operator
$$
P := \Delta_g + V,
$$
where $V\in \cinf_0(X)$ with $\supp(V) \subset K$.  
We denote by $\calR_P$ the resonance set associated to $P$.  These resonances
are the poles of the analytically continued resolvent 
$$
R_P(s) := (P - s(1-s))^{-1},
$$ 
counted according to multiplicity.  The associated resonance counting function is
$$
N_P(t) := \#\bigl\{\zeta \in \calR_P:\> |\zeta - \tfrac12| \le t\bigr\}.
$$
Our context is essentially that of Guillop\'e-Zworski \cite{GZ:1995a, GZ:1997}, and 
so we already know that $N_P(t) \asymp t^2$  (see \S\ref{res.sec} for details).  
It is thus natural to define the regularized counting function $\tN_P(a)$ just as in (\ref{tN.def}).

Before stating the upper bound, we introduce the asymptotic constants associated to the
resonance count for isolated hyperbolic planar or funnel ends. 
\begin{theorem}\label{Yj.asymp}
For a hyperbolic planar or funnel end $Y \cong [b,\infty)\times S^1$, 
with metric as in (\ref{hp.end}) or (\ref{f.end}), 
the resonance counting function for the Laplacian with Dirichlet boundary conditions at $r = b$
satisfies an asymptotic as $t \to \infty$,
$$
N_Y(t) \sim A(Y)t^2.
$$
\end{theorem}

We will write these constants $A(Y)$ explicitly in a moment.  But first let us state the main result
of this paper.

\begin{theorem}\label{main.thm}
For $(X,g)$ a surface with hyperbolic ends as in (\ref{X.KYC}),  and $V \in \cinf_0(X)$, 
the regularized counting function for  $P = \Delta_g + V$ satisfies
\begin{equation}\label{main.upper}
\frac{\tN_P(a)}{a^2} \le 
\frac{1}{2\pi} \vol \bigl(X_\rc, g\bigr) +  \sum_{j=1}^\nf A(Y_j)  + o(1),
\end{equation}
where $X_\rc$ is the subset (\ref{Xc.def}).
\end{theorem}

If $(X,g)$ is a finite-area surface with hyperbolic cusp ends (and arbitrary metric in the interior)
Parnovski \cite{Parnovski:1995} proved that
$$
N_{g}(t) \sim \frac{1}{2\pi} \vol(X, g) t^2.
$$
This shows that Theorem~\ref{main.thm} is sharp in the case $\nf = 0$.  
It also suggests an intriguing interpretation of the constants appearing in (\ref{main.upper}).  
Suppose we split $X$ into a disjoint union $X_\rc \cup Y_1 \cup \dots \cup Y_\nf$ at the 
boundary of $X_\rc$ and impose Dirichlet boundary conditions at the newly created 
boundaries.  The constant on the right-hand side of (\ref{main.upper}) is the sum of the
asymptotic constants for the resonance counting function of the resulting components.  

To obtain Theorem~\ref{hsurf.thm} from Theorem~\ref{main.thm},
we take the $Y_j$'s to be standard funnels with boundaries at $b_j=0$.  As mentioned above,
$A(Y_j) = \ell_j/4$ in that case.  And since 
$X_\rc$ has geodesic boundary and hyperbolic interior, Gauss-Bonnet gives 
$\vol(X _\rc, g) = -2\pi \chi(X)$.

As in Corollary~\ref{lb.cor}, combining Theorem~\ref{main.thm} with the Guillop\'e-Zworski argument
gives a lower bound on $N_P(t)$ with a constant that depends only on $\ovol(X, g)$ and the
end parameters $\ell_j$ and $b_j$ for $j = 1, \dots, \nf$, assuming that $\ovol(X, g)\ne 0$.

The asymptotic constants $A(Y)$ appearing
in Theorem~\ref{Yj.asymp} have a somewhat complicated form.  
Consider first a model funnel end $F_{\ell, r_0}$ defined by
\begin{equation}\label{Fell.def}
F_{\ell, r_0} \cong [r_0, \infty) \times S^1, \quad ds^2 = dr^2 +  \ell^2 \cosh^2r \frac{d\theta^2}{(2\pi)^2}.
\end{equation}
The case $r_0 = 0$, a standard funnel with geodesic boundary, is simply denoted by $F_\ell$.
The resonance set for the Laplacian on $F_{\ell, r_0}$ with Dirichlet boundary conditions
at $r = r_0$ is denoted $\calR_{F_{\ell, r_0}}$.

In \S\ref{trfun.sec} we will show that for $r_0 \ge 0$, 
\begin{equation}\label{A.f}
A(F_{\ell, r_0}) =  - \frac{\ell}{2\pi} \sinh r_0 +  \frac{4}{\pi} \int_{0}^{\frac{\pi}2} \int_0^\infty 
\frac{[I(xe^{i\theta},\ell, r_0)]_+}{x^3}\>dx\>d\theta,
\end{equation}
where $[\cdot]_+$ denotes the positive part and, with $\omega := 2\pi/\ell$,
\begin{equation}\label{Idef}
\begin{split}
I(\alpha,\ell,r) & := \re\left[ 2 \alpha \log \left( \frac{\alpha\sinh r + \sqrt{\omega^2 + \alpha^2 \cosh^2 r}}{\sqrt{\omega^2 
+ \alpha^2}} \right) \right]  \\
&\qquad + \omega \arg \left( \frac{\sqrt{\omega^2 + \alpha^2 \cosh^2 r} - i\omega \sinh r}{\sqrt{\omega^2 
+ \alpha^2 \cosh^2 r} + i \omega \sinh r} \right) + \pi (\im \alpha - \omega) .
\end{split}
\end{equation}
(The principal branch of log is used in all such formulas.)
The integral in (\ref{A.f}) is explicitly computable in the case $r_0 = 0$, 
since $I(xe^{i\theta}, \ell, 0) = \pi(x \sin\theta - \omega)$.
In this case we recover the asymptotic constant for the standard funnel, $A(F_{\ell}) = \ell/4$.

\begin{figure} 
\psfrag{rp}{$r_0=1$}
\psfrag{r0}{$r_0=0$}
\psfrag{rm}{$r_0=-1$}
\psfrag{10}{$10$}
\begin{center}  
\includegraphics{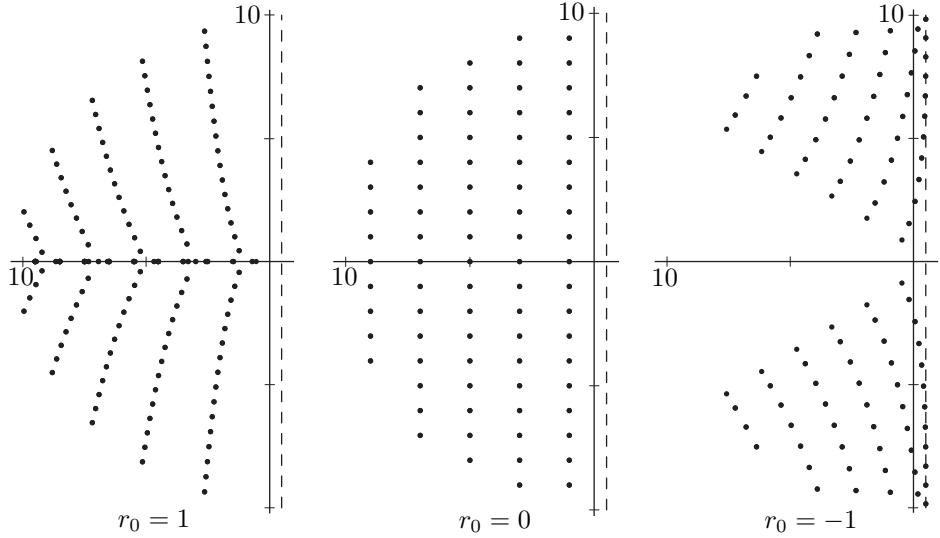} 
\end{center}
\caption{Resonance sets of the funnel $F_{\ell, r_0}$ with different boundary locations $r_0$, 
shown for $\ell = 2\pi$.}\label{extrfun}
\end{figure}
It is interesting to compare the resonance sets of truncated funnels $F_{\ell, r_0}$ with $r_0 > 0$
to extended funnels with $r_0 < 0$.  The two cases are quite different in terms of 
the classical dynamics; an extended funnel contains a trapped geodesic, while truncated funnels
are non-trapping.   Because of this change in dynamics, we expect the 
distribution of resonances near the critical line to 
change dramatically as $r_0$ switches from positive to negative.  
Figure~\ref{extrfun} illustrates these differences.  In the non-trapping case,
on the left, the distance from the resonances to the critical line increases logarithmically
as $\im s\to \infty$.   For the trapping case, on the right, the distance to the critical line 
decreases exponentially.  These behaviors are consistent with results on 
resonance-free regions for asymptotically hyperbolic manifolds by Guillarmou 
\cite{Gui:2005c}.

Of course, the asymptotics of the global counting function $N_P(t)$ are not expected 
to be sensitive to the dynamics.
Indeed, we will show in \S\ref{exfun.sec} that the formula (\ref{A.f}) for the asymptotic constant 
of $N_{F_{\ell, r_0}}(t)$ remains valid for $r_0 <0$.
This exact asymptotic can be compared to the upper bound obtained for the extended
funnel from Theorem~\ref{main.thm}, which is
\begin{equation}\label{fell.ub}
\frac{\tN_{F_{\ell, r_0}}(a)}{a^2} \le -\frac{\ell}{2\pi} \sinh r_0 + \frac{\ell}4,\quad\text{for }r_0 \le 0.
\end{equation}
Figure~\ref{Aplot} illustrates the diffierence between the upper bound (\ref{fell.ub}) and the sharp asymptotic
in this situation.  Given this discrepancy, one might think that the bound in
Theorem~\ref{main.thm} could be improved by moving the boundary of $K$ further into the interior 
of the surface (i.e. by allowing $b_j <0$ in the definition (\ref{f.end})).
Unfortunately, for reasons that we will explain in \S\ref{scdet.sec}, it does not seem possible to obtain 
any improvement this way.
\begin{figure} 
\psfrag{A}{$A(F_{\ell, r})$}
\psfrag{r}{$r$}
\psfrag{l4}{$\ell/4$}
\psfrag{lsinh}{$-\frac{\ell}{2\pi} \sinh r + \frac{\ell}4$}
\psfrag{1}{$1$}
\psfrag{2}{$2$}
\psfrag{4}{$4$}
\psfrag{-1}{$-1$}
\begin{center}  
\includegraphics{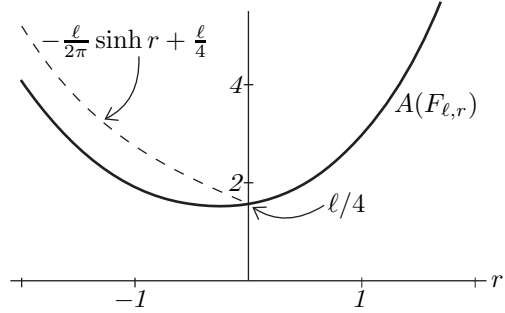} 
\end{center}
\caption{The exact asymptotic constant for $F_{\ell, r}$ as a function of boundary location $r$, 
shown for $\ell = 2\pi$.  The dotted line shows the bound from Theorem~\ref{main.thm}.}\label{Aplot}
\end{figure}

In the hyperbolic planar case, the model problem for $Y_j$ is scattering by a spherical obstacle 
in $\bbH^2$, i.e. on the exterior Dirichlet domain $\Omega_{r_0} := \{r \ge r_0\}\subset \bbH^2$.  
The resonance asymptotics for this spherical obstacles in $\bbH^{n+1}$ were worked out in 
Borthwick \cite[Thm.~1.2]{Borthwick:2009}.  In two dimensions the result is
\begin{equation}\label{A.hp}
A(\Omega_{r_0}) = 2 - \cosh r_0 + \frac{4}{\pi} \int_{0}^{\frac{\pi}2} \int_0^\infty 
\frac{[H(xe^{i\theta}, r_0)]_+}{x^3}\>dx\>d\theta,
\end{equation}
where 
\begin{equation}\label{Hdef}
\begin{split}
H(\alpha, r) & :=   \re \left[ 
2\alpha \log \left(\frac{\alpha \cosh r +  \sqrt{1 + \alpha^2 \sinh^2 r}}{\sqrt{\alpha^2 - 1}} \right) \right] \\
&\qquad + \log \left| \frac{\cosh r - \sqrt{1 + \alpha^2 \sinh^2 r}}{\cosh r + \sqrt{1 + \alpha^2 \sinh^2 r}} \right|.
\end{split}
\end{equation}

The paper is organized as follows.  The basic material on the
resolvent and resonances of the operator $P$ is reviewed in \S\ref{res.sec}. 
In \S\ref{det.sec} we present the factorization 
formula for the relative scattering determinant and show that this leads to Weyl asymptotics 
for the scattering phase
and a counting formula for resonances based on contour integration.  The growth estimates on 
the scattering determinant and the resulting proof of Theorem~\ref{main.thm} are given
in \S\ref{scdet.sec}, assuming certain estimates to be developed in later sections.
The derivation of Corollary~\ref{lb.cor} is also given in \S\ref{scdet.sec}.
In \S\ref{fmode.sec}, we develop the asymptotic analysis of Dirichlet eigenmodes on hyperbolic 
funnels.  These asymptotics are applied in \S\ref{fdet.sec} to prove the Poisson operator estimates 
needed for \S\ref{scdet.sec}.  Finally, in \S\ref{trfun.sec} and \S\ref{exfun.sec} we establish
the exact asymptotic constant (\ref{A.f}) for the truncated and extended funnel cases, respectively,
establishing the funnel part of Theorem~\ref{Yj.asymp} in particular.

\vskip12pt\noindent
\textbf{Acknowledgment.}   I would like to thank to Plamen Stefanov for suggesting the extension
of his results to the hyperbolic setting.   I am also grateful for support from
the Banff International Research Station, where some of the work for this project was done.

\section{Resonances}\label{res.sec}

The context introduced in \S\ref{intro.sec}
differs from that of Guillop\'e-Zworski \cite{GZ:1995a, GZ:1997} in two relatively minor ways:
hyperbolic planar ends are allowed in addition to funnels, and
a compactly supported potential $V$ is possibly added to $\Delta_g$.  
The latter addition really is trivial, but the inclusion of hyperbolic planar ends 
requires a few extra estimates on model terms. In this section we will briefly review the theory
\cite{GZ:1995a, GZ:1997}, in order to explain those additional estimates.

First of all, to define resonances we need analytic continuation of the resolvent, 
$R_P(s) := (P - s(1-s))^{-1}$ from its original domain $\re s > \tfrac12$.
Each end $Y_j$ is isometric to a portion of either $\bbH$ or the model funnel $F_{\ell_j}$,
and we can use this identification to pullback model resolvents $R^0_{Y_j}(s)$. 
After appropriate cutoffs are applied, we can treat these model terms as operators on $X$,
whose kernels have support only in the corresponding ends $Y_j$.  Similarly,
we define $R^0_{C_j}(s)$ by pullback from the model cusp.  
Suppose that $\chi^j_k \in \cinf(X)$ are cutoff functions for $j = 1, \dots, \nf+\nc$ and
$k= 0,1,2$, such that 
$$
\chi^j_k = \begin{cases}0 &\text{for }r\ge k+1 \text{ in end }j,\\ 
1 &\text{for }r\le k \text{ in end }j,\\
1 &\text{ outside of end }j.
\end{cases}
$$
We also set $\chi_k := \prod_j \chi^j_k$.

For some $s_0$ with $\re s_0$ sufficiently large, so that $R_P(s_0)$ is defined,
we set
$$
M(s) := \chi_2 R_P(s_0) \chi_1 + \sum_{j=1}^\nf (1-\chi^j_0) R^0_{Y_j}(s) (1 - \chi^j_1)
+ \sum_{j=\nf+1}^{\nf+\nc} (1-\chi^j_0) R^0_{C_j}(s) (1 - \chi^j_1).
$$
This parametrix satisfies
$$
(P - s(1-s)) M(s) = I - L(s),
$$
where
\[
\begin{split}
L(s) & := - [\Delta_g, \chi_2] R_P(s_0) \chi_1 + (s(1-s) - s_0(1-s_0)) \chi_2 R_P(s_0) \chi_1 \\
&\qquad + \sum_{j=1}^\nf [\Delta_g,\chi^j_0] R^0_{Y_j}(s) (1 - \chi^j_1)
+ \sum_{j=\nf+1}^{\nf+\nc} [\Delta_g, \chi^j_0] R^0_{C_j}(s) (1 - \chi^j_1).
\end{split}
\]

The are two differences here from the construction of \cite{GZ:1995a}.
First of all, some of our model terms $R^0_{Y_j}(s)$ will be copies
of $R_\bbH(s)$ instead of the funnel resolvent.  Second, we follow the treatment in
Borthwick \cite{Borthwick} in using the model resolvent for a full cusp, 
rather than modifying the original Hilbert space.  

Let $\rho \in \cinf(X)$ be proportional to $e^{-r}$ in the ends $Y_j$ and $C_j$,
with respect to the coordinate systems given in (\ref{hp.end}--\ref{c.end}).
The operator $L(s)$ is compact on $\rho^N L^2(X, dg)$ for $\re s > \tfrac12-N$ and defines
a meromorphic family with poles of finite rank.  (The structure of the kernel of 
$R^0_{Y_j}(s)$ at infinity is the same whether $Y_j$ is a funnel or hyperbolic planar, so this
part of the argument is unaffected by the addition of hyperbolic planar ends.)

By choosing $s$ and $s_0$ appropriately we can
insure that $I - L(s)$ is invertible at some $s$, and then the analytic Fredholm yields
\begin{equation}\label{RP.param}
R_P(s) = M(s) (I-L(s))^{-1}.
\end{equation}
This proves the following result, a slight generalization of \cite[Thm.~1]{GZ:1995a}:
\begin{theorem}[Guillop\'e-Zworski]\label{RP.mero}
The formula (\ref{RP.param}) defines a meromorphic extension of $R_P(s)$ to 
a bounded operator on $\rho^N L^2(X, dg)$ for $\re s > \tfrac12-N$, with poles
of finite rank.
\end{theorem}

Meromorphic continuation allows us to define
$\calR_P$ as the set of poles of $R_P(s)$, listed according to multiplicities 
given by
$$
m_P(\zeta) := \rank \res_\zeta R_P(s).
$$
The same parametrix construction also leads to an estimate of the order of growth
of the resonance counting function.  The following is a slight generalization of \cite[Thm~2]{GZ:1995a}:
\begin{theorem}[Guillop\'e-Zworski]\label{NP.bound}
The resonance counting function satisfies a bound
$$
N_P(t) = O(t^2).
$$
\end{theorem}
Our version requires just a few additional estimates.
To obtain this bound on the counting function, Guillop\'e-Zworski \cite{GZ:1995a} introduced a
Fredholm determinant 
$$
D(s) := \det (I - L_3(s)^3),
$$
where
$$
L_3(s) := L(s)\chi_3.
$$
Using the relation 
$$
R_P(s)\chi_3 = M(s) \chi_3 (I + L_3(s) + L_3(s)^2) (I - L_3(s)^3)^{-1},
$$
and a result of Vodev \cite[Appendix]{Vodev:1994}, they showed that $\calR_P$ is included in the 
union of the set of poles of $D(s)$ with 3 copies of the union of the sets of poles of $M(s)$ and $L_3(s)$. 

The only change that the inclusion of hyperbolic planar ends requires in this argument is
that for each hyperbolic planar end we include a copy of $\calR_\bbH$ among the
possible poles of $M(s)$ and $L_3(s)$.  Since $N_\bbH(t) = O(t^2)$, just as for funnels, the 
problem reduces as in \cite{GZ:1995a} to an estimate of the growth of $D(s)$.
Through Weyl's inequality, the estimate of $D(s)$ 
is broken up into estimates on the singular values of various model terms.  
We must check that the relevant estimates are satisfied by the hyperbolic planar
model terms.

There are three estimates to consider.  The first concerns the resolvent $R_\bbH(s)$.
If $Q_1, Q_2$ are compactly supported differential operators of orders $q_1, q_2$, 
with disjoint supports, then for $\vep>0$,
\begin{equation}\label{RH.est1}
\norm{Q_1 R_\bbH(s) Q_2} \le C(q_j, \vep) \>\brak{s}^{q_1+q_2},\quad\text{for }\re s > \vep,
\end{equation}
and
\begin{equation}\label{RH.est2}
\norm{Q_1 R_\bbH(s) Q_2} \le C(q_j, \vep) \>\brak{s}^{q_1+q_2-1},\quad\text{for }\re s > \tfrac12 +\vep.
\end{equation}
To prove either of these, one can simply use the explicit formula,
$$
R_\bbH(s;z,z') = \frac1{4\pi} \int_0^1 \frac{(t(1-t))^{s-1}}{[t + \sinh^2 d(z,z')]^s}\>dt,
$$
and repeat the argument from \cite[Lemma~3.2]{GZ:1995a}.

The next estimate is for the Poisson kernel $E_\bbH(s)$.  In the Poincar\'e ball model $\bbB$,
this kernel is given by
$$
E_\bbB(s;z,\theta) = \frac{1}{4\pi}  \frac{\Gamma(s)^2}{\Gamma(2s)} \frac{(1-|z|^2)^s}{|e^{i\theta} - z|^{2s}},
\qquad z \in \bbB, \>\theta \in \bbR/(2\pi\bbZ).
$$
Given a compact set $K \subset \bbB$
and $\vep>0$, we have
\begin{equation}\label{EH.est}
\bigl|\del^k_\theta E_\bbB(s; z, \theta)\bigr| \le C(K,\vep)^k \>k! \>e^{c\brak{s}}, \quad\text{for }z \in K, \> k \in \bbN.
\end{equation}
This is not difficult to prove directly by induction, or one can use an analyticity argument as in
\cite[Lemma~3.1]{GZ:1995a}.

Finally, we must estimate the scattering matrix $S_\bbH(s)$.  We can write this explicitly 
in terms of Fourier modes, 
$$
S_\bbH(s) = \sum_{k\in\bbZ} [S_\bbH(s)]_k e^{ik(\theta-\theta')},
$$
where
$$
[S_\bbH(s)]_k = 2^{1-2s} \frac{\Gamma(\tfrac12 - s)}{\Gamma(s - \tfrac12)}  \frac{\Gamma(s+|k|)}{\Gamma(1-s+|k|)}
$$
Using Stirling's formula, it is easy to use this expression for the eigenvalues to estimate the singular values
of $S_\bbH(s)$.  Assuming that $\re s < \tfrac12 - \vep$ and $\dist(s, -\bbN_0) > \eta$, we have
\begin{equation}\label{SH.est}
\mu_j(S_\bbH(s)) \le \exp \left[C(\eta)\> \brak{s} + \re(1-2s) \log \frac{\brak{s}}{j}  \right].
\end{equation}
This is the analog of \cite[Lemma~4.2]{GZ:1997}.  

With these model estimates in place, one can simply apply Guillop\'e-Zworski's
original argument (treating the cusp contributions as in \cite[\S9.4]{Borthwick}) to prove that
$$
|g(s) D(s)| \le e^{C \brak{s}^2},
$$
where $g(s)$ is a entire function of order 2 and finite type, with zeros derived from 
$\calR_\bbH$ and the the model resolvent sets for the funnels and cusps.
This yields the proof of Theorem~\ref{NP.bound}.

\section{Relative scattering determinant}\label{det.sec}

To define scattering matrices, we will fix a function $\rho \in \cinf(X)$ which serves as a
boundary defining function for a suitable compactification of $X$.  We start with smooth
positive functions $\rho_\rf$, $\rho_\rc$ satisfying
$$
\rho_\rf = \begin{cases}2e^{-r} & \text{in each }Y_j, \\
1 & \text{in each }C_j,  \end{cases}
\qquad
\rho_\rc = \begin{cases}1 & \text{in each }Y_j, \\
e^{-r} & \text{in each }C_j.  \end{cases}
$$
Then we set $\rho = \rho_\rf \rho_\rc$ for the global boundary defining function.
 
The ends $Y_j$ are conformally compact, and we distinguish between the internal boundary $\del Y_j$,
and the boundary at infinity $\diY_j$ induced by the conformal compactification.  The funnel ends $Y_j$
come equipped with a length parameter $\ell_j$, the length of the closed geodesic bounding the finite end.
If we assign length $\ell_j = 2\pi$ to a hyperbolic planar end, for consistency, then the metric 
induced by $\rho^{2} g$ on the boundary of $Y_j$ at infinity gives an isometry
$$
\diY_j \cong \bbR/\ell_j \bbZ.
$$

The cusp ends can be compactified naturally by lifting to $\bbH$ and invoking
the Riemann-sphere topology, as described in \cite[\S6.1]{Borthwick}.  The resulting boundary
$\diC_j$ consists of a single point.  

Despite the discrepancy in dimensions, it will be convenient to group all of the infinite boundaries
together as
$$
\diX := \diY_1 \cup \dots \cup \diY_{\nf} \cup \diC_{\nf+1} \cup \dots \cup \diC_{\nf+\nc}.
$$
Then we have
$$
\cinf(\diX) :=  \cinf(\bbR/\ell_1 \bbZ)\oplus \dots \oplus \cinf(\bbR/\ell_{\nf} \bbZ) \oplus \bbC^{\nc},
$$
and similarly for $L^2(\diX)$.

In \S\ref{res.sec}, $R^0_{Y_j}(s)$ denoted the pullback of the model resolvent 
in the parametrix construction.  Carrying on with this notation, we also define the model Poisson operators,
$$
E^0_{Y_j}(s): \cinf(\diY_j) \to L^2(Y_j),
$$
and scattering matrices, 
$$
S^0_{Y_j}(s): \cinf(\diY_j) \to \cinf(\diY_j).
$$
Similarly, for the cusp ends we have the Poisson kernels 
$$
E^0_{C_j}(s): \bbC \to L^2(C_j).
$$
There is no analog of the model scattering matrix for a cusp; see \cite[\S7.5]{Borthwick}
for an explanation of this.

The scattering matrix $S_P(s)$ is defined as a map on $C^\infty(\diX)$, which we
can write as
\begin{equation}\label{SP.block}
S_P(s) = \begin{pmatrix} S^\rff(s) & S^\rfc(s) \\ S^\rcf(s) & S^\rcc(s) \end{pmatrix},
\end{equation}
where the blocks are split between the `funnel-type' ends $Y_j$ and the cusps $C_j$.
The block $S^\rff(s)$ is a matrix of pseudodifferential operators; all other blocks have finite rank.
To define a scattering determinant, we normalize using the background operator
$$
S_0(s) = \begin{pmatrix} S^0_Y(s) & 0 \\ 0 & I \end{pmatrix},
$$
where 
$$
S^0_Y(s) = S^0_{Y_1}(s) \oplus \dots \oplus S^0_{Y_{\nf}}(s).
$$

The relative scattering determinant is then defined by
\begin{equation}\label{tau.def}
\tau(s) = \det S_P(s) S_0(s)^{-1}.
\end{equation}
The poles of the background scattering matrix $S_0(s)$ define a background resonance
set 
\begin{equation}\label{R0.def}
\calR_0 = \bigcup_{j=1}^\nf 
\begin{cases}\calR_{F_{\ell_j}} & \text{for a funnel end,} \\
\calR_{\bbH} &  \text{for a hyperbolic planar end.}    \end{cases}
\end{equation}
For $* = 0$ or $P$ let $H_*(s)$ denote the Hadamard product over $\calR_*$,
$$
H_*(s) := \prod_{\zeta \in \calR_*} \left(1 - \frac{s}{\zeta}\right) e^{\frac{s}{\zeta} + \frac{s^2}{2\zeta^2}}.
$$
Theorem~\ref{NP.bound} implies that the product for $H_P(s)$ converges, and for $H_0(s)$
this is clear from the definition of $\calR_0$.

\begin{proposition}\label{tau.factor}
For $P = \Delta_g + V$, the relative scattering determinant admits a factorization
$$
\tau(s) = e^{q(s)} \frac{H_P(1-s)}{H_P(s)} \frac{H_0(s)}{H_0(1-s)},
$$
where $q(s)$ is a polynomial of degree at most 2.
\end{proposition}
\begin{proof}
If the ends $Y_j$ are all hyperbolic funnels, then Guillop\'e-Zworski \cite[Prop.~3.7]{GZ:1997}
proved the factorization formula of with $q(s)$ a polynomial of degree at most 4.  The first
part of the proof, the characterization of the divisor of $\tau(s)$ obtained in \cite[Prop.~2.14]{GZ:1997}, 
remains valid if hyperbolic planar ends are included.

To extend the more difficult part of the argument, which is the estimate that shows $q(s)$ is polynomial,
we require only the extra estimates on model terms given in (\ref{RH.est1}),
(\ref{RH.est2}), (\ref{EH.est}), and (\ref{SH.est}).  With these estimates one can easily extend
the proof of \cite[Prop.~3.7]{GZ:1997}.  We refer the reader also to \cite[\S10.5]{Borthwick},
for an expository treatment of these details.  

To see that the maximal order of $q(s)$ is 2, we could prove an estimate analogous to 
\cite[Lemma~5.2]{Borthwick:2008}.   However, we will be proving a sharper version of this estimate
later in this paper.  From the proof of Theorem~\ref{tau.int.thm}, it will follow that
for some sequence $a_i \to \infty$,
$$
\log |\tau(s)| \le O(a_i^2) \quad\text{for }|s - \tfrac12| = a_i,\>\re s \ge \tfrac12.
$$
Because the Hadamard products $H_*(s)$ have order 2, this implies a bound $|q(s)| = O(|s|^{2+\vep})$
for $\re s \ge \tfrac12$.  This implies $q(s)$ has degree at most 2, since it 
is already known to be polynomial.  (The derivations leading to Theorem~\ref{tau.int.thm} 
require only that $q(s)$ is polynomial, so this argument is not circular.)
\end{proof}

To apply the factorization of $\tau(s)$ to resonance counting we introduce the 
relative scattering phase of $P$, defined as
\begin{equation}\label{sigma.def}
\sigma(\xi) := \frac{i}{2\pi} \log \tau(\tfrac12+i\xi),
\end{equation}
with branches of the log chosen so that $\sigma(\xi)$ is continuous and $\sigma(0) = 0$.
By the properties of the relative scattering matrix, $\sigma(\xi)$ is real and $\sigma(-\xi) = - \sigma(\xi)$.

To state the relative counting formula, we let $N_0$ denote the counting function associated to $\calR_0$,
$$
N_0(t) := \#\bigl\{\zeta \in \calR_0:\>|\zeta-\tfrac12| \le t\bigr\},
$$
and $\tN_0(a)$ the corresponding regularized counting function.

\begin{corollary}\label{relcount}
As $a\to \infty$,
\begin{equation}\label{tau.count}
\tN_P(a) - \tN_0(a) = 4\int_0^a \frac{\sigma(t)}{t}\>dt + 
\frac{2}{\pi} \int_{0}^{\frac\pi2} \log |\tau(\tfrac12+ ae^{i\theta})|\>d\theta + O(\log a).
\end{equation}
\end{corollary}

\noindent
The proof is by contour integration of $\tau'/\tau(s) $ around a half-circle centered at
$s = \tfrac12$.  See \cite[Prop.~3.2]{Borthwick:2009} for the details of the
derivation of (\ref{tau.count}) from Proposition~\ref{tau.factor}.
This is the analog of a formula developed by Froese \cite{Froese:1998} for Schr\"odinger 
operators in the Euclidean setting.  

The other consequence we need from Proposition~\ref{tau.factor} is essentially also 
already proven.  To analyze the first term on the right-hand side of (\ref{tau.count}),
we will invoke the Weyl-type asymptotics satisfied by the scattering phase:
\begin{theorem}[Guillop\'e-Zworski]\label{scphase.thm}
As $\xi \to +\infty$,
$$
\sigma(\xi) = \left(\frac{1}{4\pi} \ovol(X,g) - \frac{n_{\rm hp}}2 \right) \>\xi^{2} - \frac{\nc}{\pi} \xi \log \xi + O(\xi),
$$
where $n_{\rm hp}$ denotes the number of $Y_j$'s that are hyperbolic planar.
\end{theorem}
For surfaces with hyperbolic funnel or cusp ends this result was established
in Guillop\'e-Zworski \cite[Thm.~1.5]{GZ:1997}.  As in the other cases discussed above,
the modifications needed to adapt the proof to our slightly more general setting are fairly simple.
The first point is that the addition of a compactly supported potential
$V$ does not change the argument at all, since it does not affect the leading term in the
wave trace asymptotics as derived in \cite[Lemma~6.2]{GZ:1997}.
The second issue is that we allow hyperbolic planar ends in addition to funnels.
However, for $|t| < \ell$ the restriction to the diagonal of the wave kernel on a model funnel
$F_\ell$ is identical to that of $\bbH^2$.  This is the content of \cite[eq.~(6.1)]{GZ:1997}.
So hyperbolic planar ends may also be included without modifying the argument.  
Such ends do affect the final calculation, however, because $\ovol(\bbH^2) = -2\pi$ whereas
the model funnels had $\ovol(F_\ell) = 0$.  This difference accounts for the $n_{\rm hp}$ term.

\section{Scattering determinant asymptotics}\label{scdet.sec}

To state the asymptotic estimate for the scattering determinant contribution to the resonance
counting formula (\ref{tau.count}), we introduce the following constants.
If $Y_j$ is a funnel with parameters $\ell_j, b_j$, then we set
$$
B(Y_j) := \frac{4}{\pi} \int_{0}^{\frac{\pi}2} \int_0^\infty \frac{[I(xe^{i\theta},\ell_j, b_j)]_+}{x^3}\>dx\>d\theta
- \frac{\ell_j}{4},
$$
where $I(\alpha, \ell, r)$ was defined in (\ref{Idef}).
If $Y_j$ is a hyperbolic planar end with parameter $b_j$, then
$$
B(Y_j) := \frac{4}{\pi} \int_{0}^{\frac{\pi}2} \int_0^\infty \frac{[H(xe^{i\theta},b_j)]_+}{x^3}\>dx\>d\theta,
$$
where $H(\alpha, \ell, r)$ was defined in (\ref{Hdef}).
The cusps do not contribute 
to the asymptotics of $\tau(s)$ to leading order, so we make no analogous definition for $C_j$.

\begin{theorem}\label{tau.int.thm}
For $(X,g)$ a surface with hyperbolic ends as in (\ref{X.KYC}), there exists an
unbounded set $\Lambda \subset [1,\infty)$ such that
$$
\frac{2}{\pi} \int_{0}^{\frac{\pi}2} \log |\tau(\tfrac12 + ae^{i\theta})|\>d\theta 
\le \sum_{j=1}^\nf B(Y_j) a^2 + o(a^2)
$$
for all $a \in \Lambda$.  
\end{theorem}

Before undertaking the proof of Theorem~\ref{tau.int.thm}, we will show how this theorem leads to the
proof of the main result stated in \S\ref{intro.sec}:
\begin{proof}[Proof of Theorem~\ref{main.thm}]
Starting from the counting formula from Corollary~\ref{relcount}, we apply Theorem~\ref{scphase.thm}
to the scattering phase term and Theorem~\ref{tau.int.thm} to the scattering determinant contribution.
This yields
\begin{equation}\label{np.n0}
\tN_P(a) \le \tN_0(a) + \frac{1}{2\pi} \ovol(X,g) a^2 + \sum_{j=1}^\nf B(Y_j) a^2 + o(a^2),
\end{equation}
as $a \to \infty$.   From the explicit definition (\ref{R0.def}) of $\calR_0$, we see that
$$
\frac{N_0(t)}{t^2} \sim \sum_{j=1}^\nf \begin{cases} 1 & \text{for a hyperbolic planar end,} \\
\frac{\ell_j}4 & \text{for a funnel end,}  \end{cases}
$$
and so $\tN_0(a)$ satisfies the same asymptotic.  Also, we have
$$
\ovol(X,g) = \vol(X_\rc, g) + \sum_{j=1}^{\nf} \ovol(Y_j, g).
$$
The $0$-volumes of the $Y_j$'s are easily computed.  For a hyperbolic planar end,
$$
\ovol(Y_j, g) = 2\pi \FPe \int_{b_j}^{\log(2/\vep)} \sinh r\>dr = -2\pi \cosh b_j,
$$
and for a funnel end,
$$
\ovol(Y_j, g) = \ell_j \FPe \int_{b_j}^{\log(2/\vep)} \cosh r\>dr = -\ell_j \sinh b_j.
$$
By the formulas (\ref{A.hp}) and (\ref{A.f}) for $A(Y_j)$, we then see that (\ref{np.n0})
is equivalent to the claimed estimate.
\end{proof}

The derivation of Theorem~\ref{hsurf.thm} from Theorem~\ref{main.thm} was already
explained in \S\ref{intro.sec}.  To prove Corollary~\ref{lb.cor} we simply recall a few details of
the proof of the lower bound in Guillop\'e-Zworski \cite[Thm.~1.3]{GZ:1997}.  For a test function
$\phi\in \cinf_0(\bbR_+)$ with $\phi \ge 0$ and $\phi(1)>0$, we have estimates
$$
|\hat\phi(\xi)| \le C_k(1+|\xi|)^{-k-2},
$$
for $k\in\bbN$ and $\im \xi \le 0$.  Pairing the distributional Poisson formula \cite[Thm.~5.7]{GZ:1997}
with $\lambda \phi(\lambda\, \cdot)$ yields
$$
|\ovol(X,g)|\>\lambda^2 \le C_k \int_0^\infty (1+r)^{-k-3} N_P(\lambda r)\>dr.
$$
If we have $N_P(t) \le At^2$ for $t \ge 1$, then splitting the integral at $a$ gives
$$
|\ovol(X,g)|\>\lambda^2 \le C_k \Bigl[N(\lambda a) + A\lambda^2 a^{-k}\Bigr].
$$
Setting $t = \lambda a$, we have
$$
N(t) \ge \Bigl( c_k\> |\ovol(X,g)|\> a^{-2} - A a^{-2-k} \Bigr) \>t^2,
$$
and optimizing with respect to $a$ then yields
$$
N(t)  \ge c_k\> |\ovol(X,g)|^{1 + k/2} A^{-k/2}.
$$
Corollary~\ref{lb.cor} is then proven by substituting the constant obtained in Theorem~\ref{hsurf.thm}
for $A$.

\bigbreak
The rest of this section is devoted to the proof of Theorem~\ref{tau.int.thm}.
To produce a formula convenient for estimation, we introduce cutoff functions as follows.
Fix some $\eta\in (0,1)$.   For $j = 1,\dots, \nf+\nc$ and $k=1,2$, 
we define $\chi^j_k \in \cinf(X)$ so that $\chi^j_k=1$ outside the $j$-th end ($Y_j$ or $C_j$),
and inside the $j$-th end we have
\begin{equation}\label{chidef}
\chi^j_k = \begin{cases}
0 & \text{for }r \ge b_j + (k+1)\eta, \\
1 & \text{for }r \le b_j + k\eta.
\end{cases}
\end{equation}
\begin{figure} 
\psfrag{1}{$1$}
\psfrag{r}{$r$}
\psfrag{e}{$\eta$}
\psfrag{b}{$b_j$}
\psfrag{ch1}{$\chi^j_1$}
\psfrag{ch2}{$\chi^j_2$}
\begin{center}  
\includegraphics{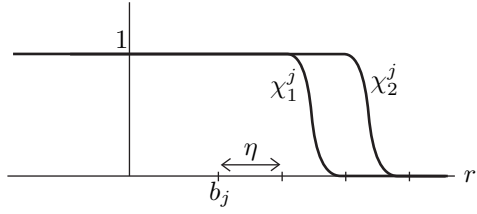} 
\end{center}
\caption{The cutoff functions $\chi^j_k$ in the $j$-th end.}\label{chijk}
\end{figure}
\begin{proposition}\label{ssiq.prop}
With cutoffs defined as in (\ref{chidef}), we have
$$
S_X(s) S_0(s)^{-1} = I + Q(s),
$$
where the components of $Q(s)$, in terms of the block decomposition introduced 
in (\ref{SP.block}), are
$$
Q^\rff_{ij}(s) = (2s-1) E^0_{Y_i}(s)^t [\Delta_{Y_i}, \chi^i_2] R_P(s)  [\Delta_{Y_j}, \chi^j_1] 
E^0_{Y_j}(1-s),
$$
$$
Q^\rcf_{ij}(s) = (2s-1) E^0_{C_i}(s)^t [\Delta_{C_i}, \chi^i_2] R_P(s)  [\Delta_{Y_j}, \chi^j_1] 
E^0_{Y_j}(1-s),
$$
$$
Q^\rfc_{ij}(s) = -(2s-1) E^0_{Y_i}(s)^t [\Delta_{Y_i}, \chi^i_2] R_P(s)  [\Delta_{C_j}, \chi^j_1] 
E^0_{C_j}(s),
$$
$$
Q^\rcf_{ij}(s) = -(2s-1) E^0_{C_i}(s)^t [\Delta_{C_i}, \chi^i_2] R_P(s)  [\Delta_{C_j}, \chi^j_1] 
E^0_{C_j}(s),
$$
\end{proposition}
\begin{proof}
We first note that one characterize the scattering matrix $S_X(s)$ through the boundary behavior of
solutions of $(\Delta_g - s(1-s))u = 0$.  For $\psi \in \cinf(\diX)$ and $\re s \ge \tfrac12$, $s \ne \bbN/2$, 
there is a unique generalized eigenfunction $u \in \cinf(X)$ with the asymptotic behavior
\begin{equation}\label{u.asymp}
u \sim \rho_\rf^{1-s} \rho_\rc^{-s} \psi + \rho_\rf^{s} \rho_\rc^{s-1} S_X(s) \psi.
\end{equation}
For hyperbolic surfaces with cusps, a proof is given in Borthwick \cite[Prop.~7.13]{Borthwick}.
The essential analysis takes place in the ends, so including smooth metric or potential perturbations 
within $K_0$ requires only trivial modifications to the proof.  Likewise, hyperbolic planar ends
may be included without much change to the argument.

Suppose $f_j \in \cinf(\diY_j)$.  Then we can use the model Poisson kernel $E^0_{Y_j}(s)$ to 
create a partial solution $(1-\chi^j_1) E^0_{Y_j}(s) f_j$ supported in $Y_j$.  As $\rho\to 0$ in $Y_j$
this function has the asymptotic behavior
\begin{equation}\label{yj.asymp}
(1-\chi^j_1) E^0_{Y_j}(s) f_j \sim \frac{1}{2s-1} \Bigl[ \rho_\rf^{1-s}f_j + \rho_\rf^{s} S^0_{Y_j}(s)f_j \Bigr].
\end{equation}
To create a full solution will take the ansatz
$$
u = (1-\chi^j_1) E^0_{Y_j}(s) f_j + u'
$$
and then solve $(\Delta_g - s(1-s))u = 0$ for $u'$ by applying the resolvent.
The result is
$$
u' = R_P(s) [\Delta_{Y_j}, \chi^j_1] E^0_{Y_j}(s) f_j.
$$
In the end $Y_i$, we can use the fact that $(1 - \chi^i_2)  [\Delta_{Y_j}, \chi^j_1] = 0$ to deduce
$$
(\Delta_{Y_i} - s(1-s))(1 - \chi^i_2) u' = - [\Delta_{Y_i}, \chi^i_2] u',
$$
and hence that
$$
(1 - \chi^i_2) u' = - R^0_{Y_i}(s) [\Delta_{Y_i}, \chi^i_2]R_P(s) [\Delta_{Y_j}, \chi^j_1] E^0_{Y_j}(s) f_j.
$$
This gives the asymptotic behavior in $Y_i$:
\begin{equation}\label{yi.asymp}
u' \sim - \rho_\rf^{s} E^0_{Y_i}(s)^t [\Delta_{Y_i}, \chi^i_2]R_P(s) [\Delta_{Y_j}, \chi^j_1] E^0_{Y_j}(s) f_j.
\end{equation}
By comparing the asymptotics (\ref{yj.asymp}) and (\ref{yi.asymp}) to the general form (\ref{u.asymp}),
we see that
$$
S^\rff_{ij}(s) = \delta_{ij} S^0_{Y_j}(s) - (2s-1)
E^0_{Y_i}(s)^t [\Delta_{Y_i}, \chi^i_2]R_P(s) [\Delta_{Y_j}, \chi^j_1] E^0_{Y_j}(s)
$$
We then obtain $Q^\rff_{ij}(s)$ by noting that
$$
E^0_{Y_j}(s) S^0_{Y_j}(s)^{-1} =  - E^0_{Y_j}(1-s).
$$

To find $Q^\rcf_{ij}(s)$ we use the same setup starting from $f_j \in \cinf(\diY_j)$, but
then analyze $u'$ by restricting to the the cusp end $C_i$.  This yields
$$
(1 - \chi^i_2) u' = - R^0_{C_i}(s) [\Delta_{C_i}, \chi^i_2]R_P(s) [\Delta_{Y_j}, \chi^j_1] E^0_{Y_j}(s) f_j.
$$
The asymptotic behavior in $C_i$ is given by 
$$
(1 - \chi^i_2) u' \sim - \rho^{s-1} E^0_{C_i}(s)^t [\Delta_{C_i}, \chi^i_2]R_P(s) 
[\Delta_{Y_j}, \chi^j_1] E^0_{Y_j}(s) f_j,
$$
so that
$$
S^\rcf_{ij}(s) = -(2s-1) E^0_{C_i}(s)^t [\Delta_{C_i}, \chi^i_2]R_P(s) [\Delta_{Y_j}, \chi^j_1] E^0_{Y_j}(s).
$$

Next take $a_j \in \cinf(\diC_j) = \bbC$.  
Since $E^0_{C_j}(s; r) = \rho_\rc^{-s}/(2s-1)$, our ansatz for a generalized eigenfunction 
satisfying (\ref{u.asymp}) starts from
$$
(1 - \chi^j_1) E^0_{C_j}(s) a_j \sim \frac{1}{2s-1} \rho_\rc^{-s} a_j.
$$
The corresponding generalized eigenfunction is
$$
u = (1 - \chi^j_1) E^0_{C_j}(s) a_j + u',
$$
where
$$
u' = R_P(s) [\Delta_{C_\infty}, \chi^j_1] E^0_{C_j}(s) a_j.
$$
arguing as above, we find that
$$
u' \sim - \rho_\rf^s E^0_{Y_i}(s)^t [\Delta_{Y_i}, \chi^i_2]R_P(s)
[\Delta_{C_\infty}, \chi^j_1] E^0_{C_j}(s) a_j
$$
in the funnel $Y_i$, and
$$
u' \sim - \rho_\rc^{1-s} E^0_{C_i}(s)^t [\Delta_{C_i}, \chi^i_2]R_P(s)
[\Delta_{C_j}, \chi^j_1] E^0_{C_j}(s) a_j
$$
in the cusp $C_i$.  We can then read off the matrix elements, $S^\rfc_{ij}(s)$ and
$S^\rcc_{ij}(s)$, as above.
\end{proof}

In conjunction with the cutoffs defined in (\ref{chidef}), we introduce projections $\chr^j_k$ 
on $L^2(X,dg)$, where
\begin{equation}\label{chr.def}
\chr^j_k f = \begin{cases}
f& \text{ for }r \in [b_j + k\eta, b_j + (k+1)\eta]\text{ in end }j.\\
0 & \text{else}.
\end{cases}
\end{equation}
As with the cutoffs, these projections depend on $b_j$ and also on the choice of
$\eta>0$.  We then introduce operators on
$L^2(X,dg)$ given by
\begin{equation}\label{Gj.funnel}
G_j(s) := (2s-1) \chr^j_1 E^0_{Y_j}(1-s) E^0_{Y_j}(s)^t \>\chr^j_2,
\end{equation}
for $j = 1, \dots, \nf$, and
\begin{equation}\label{Gj.cusp}
G_j(s) := - (2s-1) \chr^j_1 E^0_{C_j}(s) E^0_{C_j}(s)^t \>\chr^j_2,
\end{equation}
for $j = \nf+1, \dots, \nf+\nc$.
\begin{proposition}\label{tau.gj.prop}
The relative scattering phase is bounded by
$$
\log |\tau(s)| \le \sum_{j=1}^{\nf+\nc} \log \det\Bigl(I + C(\eta, \vep)\> |G_j(s)|\Bigr),
$$
for $\re s \ge \tfrac12$ with $\dist(s(1-s), \sigma(P)) \ge \vep$.
\end{proposition}
\begin{proof}
In the formula for the relative scattering matrix given in Proposition~\ref{ssiq.prop},
we can write $Q(s)$ as the composition of three operators,
$$
Q(s): L^2(\diX) \stackrel{Q_3}\longrightarrow L^2(X,dg)  \stackrel{Q_2}\longrightarrow L^2(X, dg) 
\stackrel{Q_1}\longrightarrow L^2(\diX),
$$
where
$$
Q_1 :=  \sum_{j=1}^\nf E^0_{Y_j}(s)^t \chr^j_2 +
\sum_{j = \nf+1}^{\nf+\nc} E^0_{C_j}(s)^t \chr^j_2,
$$
$$
Q_2 := \sum_{i,j = 1}^{\nf+\nc}  [\Delta_{g}, \chi^i_2] R_P(s) [\Delta_{g}, \chi^j_1],
$$
$$
Q_3\big|_{L^2(\diY_j)} := \chr^j_1 E^0_{Y_j}(1-s),\qquad 
Q_3\big|_{L^2(\diC_j)} := \chr^j_1 E^0_{C_j}(s).
$$

By the cyclicity of the trace,
$$
\tau(s) = \det (I + Q(s)) = \det(I + Q_2\circ Q_3\circ Q_1).
$$
Under the assumptions $\re s \ge \tfrac12$ with $\dist(s(1-s), \sigma(P)) \ge \vep$, we can apply the
spectral theorem and standard elliptic estimates to prove that
$$
\norm{Q_2} \le C(\eta, \vep).
$$
By the Weyl estimate this then gives
$$
|\tau(s)| \le \prod_{j=1}^\infty \bigl(1 + C(\eta, \vep) \mu_j(Q_3\circ Q_1)\bigr) = \det (1 + C(\eta, \vep) \>|Q_3\circ Q_1|\bigr)
$$
The result follows immediately from
$$
Q_3 \circ Q_1 = G_1 \oplus \dots \oplus G_{\nf+\nc},
$$
where the $G_j(s)$ are given by (\ref{Gj.funnel}) and (\ref{Gj.cusp}).
\end{proof}

Note that the right-hand side of the estimate from Proposition~\ref{tau.gj.prop} is always positive.  
It is therefore impossible to obtain a sharp estimate by this approach in cases where the leading 
asymptotic behavior of $\log |\tau(s)|$ is negative.  The extended funnel, whose resonance asymptotics 
are studied in \S\ref{exfun.sec}, gives an example of this situation.

\bigbreak
\begin{proof}[Proof of Theorem~\ref{tau.int.thm}]
Let $\calR_0$ be the background resonance set as defined in (\ref{R0.def}).
To avoid poles, we will restrict our attention to radii in the set
$$
\Lambda := \left\{ a \ge 1 :\> \dist\Bigl(\bigl\{|s-\tfrac12| = a\bigr\}, \calR_0\cup \calR_P \Bigr) \ge a^{-3} \right\}.
$$
Since $N_0(t)$ and $N_P(t)$ are $O(t^2)$, the density of $\Lambda$ in $[1,r)$ approaches
1 as $r \to \infty$.  

If we assume that $0 \le \theta \le \tfrac{\pi}2 - \vep a^{-2}$, then $s = \tfrac12 + ae^{i\theta}$ will
satisfy the hypothesis that $\dist(s(1-s), \sigma(P)) \ge \vep$ for Proposition~\ref{tau.gj.prop}.
We also assume $a \in \Lambda$ throughout this argument.
If $Y_j$ is a funnel end, then Proposition~\ref{detG.prop} gives 
\begin{equation}\label{detG.fun}
\log \det\Bigl( I + C(\eta, \vep)\>\bigl|G_j(\tfrac12 + ae^{i\theta})\bigr|\Bigr) 
\le \kappa_j(\theta, b_j+4\eta) + C(\eta,\vep,b_j)\>a \log a,
\end{equation}
where 
$$
\kappa_j(\theta, r) :=  2 \int_0^{\infty} \frac{[I( xe^{i\theta}, \ell_j, r)]_+}{x^3}\>dx
- \frac12 \ell_j \sin^2\theta, 
$$
If $Y_j$ is hyperbolic planar, the corresponding estimate follows from
Borthwick \cite[Prop.~5.4]{Borthwick:2009}, with 
$$
\kappa_j(\theta, r) :=  2 \int_0^{\infty} \frac{[H(xe^{i\theta}, r)]_+}{x^3}\>dx, 
$$
(A slight modification of the original proof is required, replacing the assumption $a \in \bbN$ with an estimate
based on $\dist(\frac12 - ae^{i\theta}, -\bbN)$.)

For a cusp end $C_j$, it is easy to estimate directly since
$$
E_{C_j}^0(s) = \frac{e^{sr}}{2s-1},
$$
which gives
$$
G_j(s; r,\theta, r',\theta') =  - \frac{1}{2s-1} \chr_{j,1}(r) e^{s(r+r')} \chr_{j,2}(r'). 
$$
This operator has rank one, so that
$$
\det\bigl( I + c\>|G_j(s)|\bigr) = 1 + \mu_1(G_j(s)),
$$
where the sole singular value is given by
$$
\mu_1(G_j(s)) = \frac{1}{|2s-1|} \left[ \int_{b_j+\eta}^{b_j+2\eta} e^{2r\re s}\>e^{-r} \>dr \right]^{\frac12}
\left[ \int_{b_j+2\eta}^{b_j+3\eta} e^{2r\re s}\>e^{-r} \>dr \right]^{\frac12}.
$$
Hence we have 
$$
 \det\Bigl( I + c\>\bigl|G_j(\tfrac12 + ae^{i\theta})\bigr|\Bigr) \le 1 + \frac{c}{2a} e^{2a (b_j+3\eta)}.
$$
For $a$ sufficiently large,
\begin{equation}\label{detG.cusp}
\log \det\Bigl( I + C(\eta, \vep)\>\bigl|G_j(\tfrac12 + ae^{i\theta})\bigr|\Bigr) \le C(\eta, \vep, b_j) a,
\end{equation}
for all $|\theta|\le \tfrac{\pi}2$.

From (\ref{detG.fun}) and (\ref{detG.cusp}) we conclude that
\begin{equation}\label{tau.fj}
\frac{\log |\tau(\tfrac12 + ae^{i\theta})|}{a^2} \le \sum_{j=1}^\nf \kappa_j(\theta, b_j+4\eta) 
+ C(\eta,\vep,b_j)\>a^{-1} \log a,
\end{equation}
for $a \in \Lambda$ and $0 \le \theta \le \tfrac{\pi}2 - \vep a^{-2}$.
Since the $\kappa_j(\theta,r)$ are uniformly continuous on $[0, \tfrac{\pi}2] \times [b_j,b_j+1]$, 
we can take $\eta \to 0$ in (\ref{tau.fj}), to obtain
\begin{equation}\label{inttau}
\frac{\log |\tau(\tfrac12 + a e^{i\theta})| }{a^2} \le \sum_{j=1}^\nf \kappa_j(\theta, b_j) +o(a^2),
\end{equation}
uniformly for $0 \le \theta \le \tfrac{\pi}2 - \vep a^{-2}$.

By integrating the estimate (\ref{inttau}) over $\theta$, we obtain
$$
\frac{2}{\pi} \int_{0}^{\frac{\pi}2 - \vep a^{-2}} \log |\tau(\tfrac12 + a e^{i\theta})| \>d\theta
\le \sum_{j=1}^\nf B(Y_j)a^2 + o(a^2).
$$
It remains to fill in the small gap where $|\theta|$ is close to $\tfrac{\pi}2$.
The factorization given by Proposition~\ref{tau.factor}, 
together with the Minimum Modulus Theorem \cite[Thm.~3.7.4]{Boas}, 
implies that for any $\eta>0$, 
\begin{equation}\label{tau.order.est}
|\tau(\tfrac12 + a e^{i\theta})| \le C_\eta \exp(a^{2+\eta}),
\end{equation}
provided $a \in J$.  (This was the reason that $\calR_P$ was included in the definition of $J$.)
Thus, 
$$
\frac{2}{\pi}\int_{\frac{\pi}2 - \vep a^{-2}}^{\frac{\pi}2}  \log |\tau(\tfrac12+ ae^{i\theta})|\>d\theta
= O(a^\eta \vep),
$$
and so this term can be absorbed into the $o(a^2)$ error.
\end{proof}

\bigbreak
To conclude this section, we'll derive some uniform upper and lower bounds on the growth of $\tau(s)$
for $s \in \bbC$, refining the estimates that one could obtain directly from Proposition~\ref{tau.factor}.
These will prove useful in \S\ref{trfun.sec} and \S\ref{exfun.sec}, in particular.
\begin{lemma}\label{lind.lemma}
Let $\mathcal{Q}$ denote the joint set of zeros and poles of $\tau(\frac12+z)$
and $\tau(\frac12 - iz)$.  Assuming $|z| \ge 1$ and $\dist(z, \mathcal{Q}) > |z|^{-\beta}$
with $\beta>2$, we have
$$
-c(\beta) |z|^2 \le \log |\tau(\tfrac12+z)| \le C(\beta) |z|^2.
$$
\end{lemma}
\begin{proof}
Since $\tau(\frac12-z) = 1/\tau(\frac12-z)$ and $\tau(\tfrac12+\bar z) = \overline{\tau(\tfrac12+z)}$,
it suffices to prove the bounds for $z$ in the first quadrant.

For $\re z \ge \delta$ with $\delta >0$, the upper bound is given in  
(\ref{tau.fj}).  As long as $\delta <1$, the function $\tau(s)$ is analytic in the strip
$\re z \in [0, \delta]$.  And since $\log |\tau(\tfrac12+z)| = 1$ for $\re z = 0$,
the bound $\log |\tau(\tfrac12+z)| = O(|z|^2)$ extends to the strip $\re z \in [0, \delta]$ 
by (\ref{tau.order.est}) and the Phragm\'en-Lindel\"of Theorem.

To prove the lower bound, consider the Hadamard products appearing in the factorization of
$\tau(s)$ given in Proposition~\ref{tau.factor}.   These products are of
order 2 but not finite type, so applying the Minimum Modulus Theorem directly would give
$-\log |\tau(\tfrac12+z)| = O(|z|^{2+\eta})$, away from the zeros.  
However, Lindel\"of's Theorem (see 
e.g.~\cite[Thm.~2.10.1]{Boas}) shows that products of the form $H_*(\tfrac12 + z)
H_*(\tfrac12 \pm iz)$ are of finite type.  In other words,
$$
\log |H_*(\tfrac12 + z) H_*(\tfrac12 \pm iz)| \le C |z|^2,
$$
as $|z| \to \infty$.  Using these estimates, and their implications via the 
Minimum Modulus Theorem \cite[Thm.~3.7.4]{Boas}, we can prove a lower bound
\begin{equation}\label{tt.bnd}
\log |\tau(\tfrac12 + z)|  \ge  - c(\beta) |z|^2 -  \log|\tau(\tfrac12 \pm iz)|,
\end{equation}
provided $\tfrac12 +z$ and $\tfrac12\pm iz$ stay at least a distance $|z|^{-\beta}$
away from the sets $1- \calR_{F_{\ell,r_0}}$ and $\calR_{F_{\ell}}$, with $\beta>2$.

Assuming $\arg z \in [0, \tfrac{\pi}2]$, 
we already know $\log |\tau(\tfrac12 - iz)| \le C(\beta) |z|^2$ from above, provided 
$\tfrac12 - iz$ stays at least a distance $|z|^{-\beta}$ away from the sets 
$\calR_{F_{\ell,r_0}}$ and $1-\calR_{F_{\ell}}$.  The lower bound in the first
quadrant then follows from (\ref{tt.bnd}).
\end{proof}

\section{Funnel eigenmodes}\label{fmode.sec}

Let $F_\ell$ be a hyperbolic funnel of diameter $\ell$.  In geodesic coordinates
$(r,\theta) \in \bbR_+\times S^1$, 
defined with respect to the closed geodesic neck, the metric is
\begin{equation}\label{fun.metric}
g_0 = dr^2 + \cosh^2 r\> \frac{d\theta^2}{\omega^2},
\end{equation}
where
$$
\omega := \frac{2\pi}{\ell}.
$$
The Laplacian is given by
\begin{equation}\label{fun.lapl}
\Delta_{F_\ell} = - \del_r^2 - \tanh r\>\del_r - \frac{\omega^2}{\cosh^2 r}\del_\theta^2.
\end{equation}
In this section we will consider asymptotic properties of the Fourier modes of generalized 
eigenfunctions of $\Delta_{F_\ell}$.

The restriction of eigenvalue equation $(\Delta_{F_\ell} - s(1-s)) u = 0$ to the $k$-th Fourier
mode, $u = w(r) e^{i k\theta}$, yields the equation,
\begin{equation}\label{kmode.eq}
-\del_r^2 w - \tanh r\>\del_r w + \left[\frac{k^2\omega^2}{\cosh^2 r} - s(1-s)\right]w = 0.
\end{equation}
This is essentially a hypergeometric equation.  With respect to the symmetry $r \mapsto -r$, we have
an even solution,
\begin{equation}\label{wp.def}
w^+_k(s; r) := (\cosh r)^{i\omega k}\> \bF(\tfrac{s+i\omega k}2,\tfrac{1-s+i\omega k}2; \tfrac12; -\sinh^2 r), 
\end{equation}
and an odd solution,
\begin{equation}\label{wm.def}
w^-_k(s; r) := \sinh r (\cosh r)^{i\omega k}\> \bF(\tfrac{1+s+i\omega k}2,\tfrac{2-s+i\omega k}2; \tfrac32; -\sinh^2 r) .
\end{equation}

By symmetry, we will always be able to assume that $k\ge 0$.
If we substitute $w = (\cosh r)^{-\frac12} U$ and introduce the parameter $\alpha$ defined by
$s = \tfrac12 + k\alpha$, the coefficient equation (\ref{kmode.eq}) becomes
\begin{equation}\label{k2fg}
\del_r^2 U = (k^2 f + g) U,
\end{equation}
where
$$
f := \frac{\omega^2 + \alpha^2 \cosh^2 r}{\cosh^2 r}, \qquad g := \frac{1}{4 \cosh^2 r}.
$$
This equation has turning points when $\alpha = \pm i \omega/\cosh r$.  We will restrict our attention to
$\arg \alpha \in [0, \tfrac\pi2]$, so that we only consider the upper turning point.   The Liouville transformation
involves a new variable $\zeta$ defined by integrating 
\begin{equation}\label{zetadiff}
\sqrt{\zeta}\>d\zeta := \sqrt{f} \>dr,
\end{equation}
on a contour that starts from the upper turning point.  
Integrating (\ref{zetadiff}) yields
\begin{equation}\label{zeta.phi}
\tfrac23 \zeta^{\frac32} = \phi,
\end{equation}
where $\phi(\alpha, r)$, the integral of $\sqrt{f} \>dr$ from the turning point, is given explicitly by
\begin{equation}\label{phi.def}
\begin{split}
\phi(\alpha, r) & := \alpha \log \left( \frac{\alpha\sinh r + \sqrt{\omega^2 + \alpha^2 \cosh^2 r}}{\sqrt{\omega^2 
+ \alpha^2}} \right) \\
&\qquad + \frac{i\omega}2 \log 
\left( \frac{\sqrt{\omega^2 + \alpha^2 \cosh^2 r} - i\omega \sinh r}{\sqrt{\omega^2 + \alpha^2 \cosh^2 r}+i\omega \sinh r} \right)
+ \phi_0(\alpha),
\end{split}
\end{equation}  
for $\alpha \ne i\omega$, where
\begin{equation}\label{phi0.def}
\phi_0(\alpha) = \phi(\alpha;0) = -\tfrac{\pi}2(i\alpha+\omega).
\end{equation}
By continuity, the definition of $\phi$ extends to $\alpha = i\omega$, with
$$
\phi(i\omega, r) = i \omega \log \cosh r.
$$

To complete the Liouville transformation, we set 
$W = (f/\zeta)^{\frac14} U$, so that the equation (\ref{k2fg}) becomes an approximate
Airy equation,
\begin{equation}\label{W.eq}
\del_\zeta^2 W = (k^2 \zeta + \psi) W,
\end{equation}
with the extra term given by
\begin{equation}\label{psi.gfz}
\psi =  \frac{\zeta}{4f^2}\del_r^2 f - \frac{5\zeta }{16f^3} (\del_r f)^2 + \frac{\zeta g}{f}
+\frac{5}{16 \zeta^2}.
\end{equation}
The solutions of (\ref{W.eq}) are of the form
\begin{equation}\label{W.sigma}
W_\sigma := \Ai(k^{\frac23}e^{\frac{2\pi i \sigma}{3}} \zeta) + h_\sigma(k, \alpha, r),
\end{equation}
where $\sigma = 0$ or $\pm 1$, and the error term satisfies the differential equation,
\begin{equation}\label{h.diffeq}
\del_\zeta^2 h_\sigma - k^2 \zeta h_\sigma = \psi\>\bigl[h_\sigma 
+ \Ai(k^{\frac23}e^{\frac{2\pi i \sigma}{3}} \zeta)\bigr].
\end{equation}

Using methods from Olver \cite{Olver} we can control this error term.
\begin{lemma}\label{hsig.bound}
The error equation (\ref{h.diffeq}) admits solutions that satisfy $\lim_{r\to \infty} h_\sigma(r) = 0$
and 
$$
|h_\sigma| \le C k^{-1}|\alpha|^{-\frac23} \bigl(1+|k\phi|^{\frac16}\bigr)^{-1} e^{(-1)^{\sigma+1} k\re \phi},
$$
with $C$ independent of $r$, $k$ and $\alpha$.
\end{lemma}

We will defer the rather technical proof of Lemma~\ref{hsig.bound} to 
the end of this section, in order to concentrate
on the implications of (\ref{W.sigma}).  
The asymptotics of the Airy function are well known (see e.g.~\cite[\S11.8]{Olver}).
Uniformly for $|\arg z| < \pi - \vep$ we have
\begin{equation}\label{ai.asym1}
\Ai(z) = \frac{1}{2\pi^\frac12} z^{-\frac14} \exp(-\tfrac23 z^{\frac32}) \>[1 + O(|z|^{-\frac32})].
\end{equation}
And uniformly for $|\arg z| \ge \tfrac{\pi}3 + \vep$,
\begin{equation}\label{ai.asym2}
\Ai(z) = \frac{1}{\pi^{\frac12}} (-z)^{-\frac14} \cos \Bigl(\tfrac23 (-z)^{\frac32}-\tfrac{\pi}4\Bigr) 
\bigl[1 + O(|z|^{-\frac32})\bigr],
\end{equation}
These asymptotics make it convenient to introduce a pair of solutions of the eigenvalue equation
(\ref{kmode.eq}) defined by 
\begin{equation}\label{wsig.def}
w_\sigma =  2\pi^\frac12  e^{\frac{i \pi \sigma}{6}} k^{\frac16} \zeta^\frac14 
(\omega^2+\alpha^2\cosh^2 r)^{-\frac14} W_\sigma,
\end{equation}
where $W_\sigma$ is the ansatz (\ref{W.sigma}) for $\sigma = 0$ or $1$.
\begin{proposition}\label{wsig.prop}
Consider the solutions of the equation,
$$
(\Delta_{F_\ell} - \tfrac14 - k^2\alpha^2)e^{ik\theta}w_\sigma(r)  = 0,
$$
given by (\ref{wsig.def}) with $\sigma = 0$ or $1$.
Assuming $k\ge 1$ and $\arg \alpha \in [0,\tfrac{\pi}2 - \vep]$, we have asymptotics
\begin{equation}\label{wsig.asymp}
w_\sigma = (\omega^2 + \alpha^2 \cosh^2 r)^{-\frac14}  \exp \bigl[(-1)^{\sigma+1}k \phi\bigr]
\bigl( 1 + O(|k\alpha|^{-1})\bigr),
\end{equation}
with constants that depend only on $\vep$.
In addition, for $\arg \alpha \in [0, \tfrac{\pi}2]$ and $|k\alpha|$ sufficiently large, we have the upper bounds
\begin{equation}\label{wsig.upper}
|w_\sigma| \le Ck^{\frac16} \exp \bigl[(-1)^{\sigma+1}k\re \phi\bigr],
\end{equation}
and the lower bound
\begin{equation}\label{w0.lower}
|w_0| \ge c e^{-k \re \phi}.
\end{equation}
\end{proposition}
\begin{proof}
The assumption that  $\arg \alpha \in [0,\tfrac{\pi}2-\vep]$, implies that 
$\arg \zeta \in [-\tfrac{2\pi}3, \tfrac{\pi}3-\vep]$, so that 
(\ref{ai.asym1}) applies to both $w_0$ and $w_1$ in this case.  It also
implies that $|\phi| \ge c(\vep) (|\alpha| + 1)$, so that the error term
$O(|w|^{-\frac32})$ from (\ref{ai.asym1}) becomes $O(|k\alpha|^{-1})$
when applied to $|w| = k^{\frac23} |\zeta|$.
In combination with Lemma~\ref{hsig.bound}, this proves (\ref{wsig.asymp}), 
and also (\ref{wsig.upper}) and (\ref{w0.lower}) in the case
where $\arg \alpha$ is bounded away from $\tfrac{\pi}2$.

If $\arg \alpha \in [\tfrac{\pi}2-\vep, \tfrac{\pi}2]$, then (\ref{ai.asym1}) and (\ref{ai.asym2}),
together with Lemma~\ref{hsig.bound}, give the estimates
\begin{equation}\label{kzW.est}
|k^{\frac16} \zeta^\frac14 W_\sigma| \le  C \exp \bigl[(-1)^{\sigma+1}k\re \phi\bigr],
\end{equation}
and
\begin{equation}\label{kzW.low}
|k^{\frac16} \zeta^\frac14 W_0| \ge  c e^{-k\re \phi},
\end{equation}
If $|\omega^2 + \alpha^2 \cosh^2r| \ge 1$, which bounds $\phi$ away from $0$, 
then this gives (\ref{wsig.upper}) immediately.
This leaves the case $|\omega^2 + \alpha^2 \cosh^2r| \ge 1$, which puts $\phi$ close to
zero.  In this case, $\zeta \asymp (\omega^2 + \alpha^2 \cosh^2r)$, so that 
$w_\sigma \asymp k^{\frac16} W_\sigma$.  Then if $|k\phi| \ge 1$ we can derive the estimates
from (\ref{kzW.est}) and (\ref{kzW.low}), while for $|k\phi| \le 1$ we simply note that $W_\sigma$ is bounded
and nonzero near the origin.
\end{proof}

Another detail we will need later is the asymptotic behavior of $w_\sigma$ as $r \to \infty$.
\begin{lemma}\label{wsig.rlim}
For $\re \alpha \ge 0$, as $r \to \infty$,
$$
w_0 \sim \alpha^{-\frac12} e^{-k(\phi_0(\alpha) + \gamma(\alpha))} \rho^{\frac12+k\alpha},
$$
and
$$
w_1 \sim \alpha^{-\frac12} e^{k(\phi_0(\alpha) + \gamma(\alpha))} \left( \rho^{\frac12-k\alpha} + 
i \rho^{\frac12+k\alpha}\right)
$$
where $\rho := 2e^{-r}$, and
\begin{equation}\label{gamma.def}
\gamma(\alpha) := \alpha \log \frac{2\alpha}{\sqrt{\omega^2+\alpha^2}} 
+ \frac{i\omega}2 \log \frac{\alpha - i\omega}{\alpha + i\omega}.
\end{equation}
\end{lemma}
\begin{proof}
The results follow immediately from (\ref{ai.asym1}) and (\ref{ai.asym2}), in combination
with the asymptotic 
\begin{equation}\label{phi.rinf}
\phi(\alpha; r) = \alpha r + \phi_0(\alpha) + \alpha \log \frac{\alpha}{\sqrt{\omega^2+\alpha^2}} 
+ \frac{i\omega}2 \log \frac{\alpha - i\omega}{\alpha + i\omega} + O(r^{-1}),
\end{equation}
as $r \to \infty$.
\end{proof}

\bigbreak
We conclude the section with the proof of the error estimate that is the basis of
Proposition~\ref{wsig.prop} and Lemma~\ref{wsig.rlim}.
\begin{proof}[Proof of Lemma~\ref{hsig.bound}]
The cases of different $\sigma$ are all very similar, so we consider only $\sigma=0$.  
In this case combining the boundary condition with variation of parameters allows us to
transform (\ref{h.diffeq}) into an integral equation:
$$
h_0(k,\alpha, r) = \frac{2\pi e^{-\frac{i\pi}6}}{k^{\frac23}}
\int_r^\infty K_0(r,r') \psi(r') \>[h_0(k,\alpha,r') + \Ai(k^{\frac23}\zeta(r'))]\>  
\frac{f(r')^\frac12}{\zeta(r')^\frac12} \>dr',
$$
where
$$
K_0(r,r') := \Ai(k^\frac23 \zeta(r'))\Ai(k^\frac23 e^{-\frac{2\pi i}3}\zeta(r)) 
- \Ai(k^\frac23 e^{-\frac{2\pi i}3}\zeta(r'))\Ai(k^\frac23 \zeta(r)).
$$
Then, using the method of successive approximations as in \cite[Thm.~6.10.2]{Olver},  
together with the bounds on the Airy function and its derivatives developed in \cite[\S 11.8]{Olver}, 
we obtain the bound,
\begin{equation}\label{h.bound}
| h_0| \le Ce^{-k \re \phi} (1+ k^{\frac16}\>|\zeta|^{\frac14})^{-1} \Bigl(e^{ck^{-1} \Psi(r)} - 1 \Bigr),
\end{equation}
where 
\begin{equation}\label{Psi.def}
\Psi(r) := \int_r^\infty \left|\psi  f^{\frac12} \zeta^{-\frac12} \right| dr'.
\end{equation}

From (\ref{psi.gfz}), we compute
\begin{equation}\label{pfz}
\begin{split}
\psi f^{\frac12} \zeta^{-\frac12} &=  \left[ \frac{\alpha^4 \cosh^2 r + 4 \alpha^2 \omega^2 \sinh^2 r  
-\omega^4}{4(\omega^2 + \alpha^2 \cosh^2 r)^{\frac52}} \right] \zeta^{\frac12} \cosh r \\
&\qquad + \frac{5}{16} \frac{(\omega^2 + \alpha^2 \cosh^2 r)^{\frac12}}{\zeta^{\frac52} \cosh r}.
\end{split}
\end{equation}
The estimate must be broken into various regions.  Fix some $c>0$.

\emph{Case 1.}  Assume $|\alpha| \ge 1$ and $|\omega^2 + \alpha^2 \cosh^2(r)|\ge c$.
Under these conditions, we can estimate
$$
|\phi| \asymp |\alpha|(r+1).
$$
Then from (\ref{pfz}), we find
$$
\left|\psi  f^{\frac12} \zeta^{-\frac12} \right| \le C_1 |\alpha|^{-\frac23} e^{-2r} (r+1)^{\frac13}
+ C_2 |\alpha|^{-\frac23} (r+1)^{-\frac53}.
$$
We easily conclude that for $|\alpha| \ge 1$,
\begin{equation}\label{phi.bound1}
\int_{|\omega^2 + \alpha^2 \cosh^2(r)|\ge c} \left|\psi  f^{\frac12} \zeta^{-\frac12} \right| dr = O(|\alpha|^{-\frac23}).
\end{equation}

\emph{Case 2.}  Assume $|\alpha| \le 1$ and $|\omega^2 + \alpha^2 \cosh^2(r)|\ge c$.
The behvarior of $\phi$ is now slightly more complicated, depending on the size of $r$ relative
to $|\alpha|$,
$$
|\phi| \asymp \begin{cases} |\alpha| + e^{-r} & \text{for } |\alpha| \sinh r \le 1, \\
|\alpha| (r+ \log|\alpha|) & \text{for } |\alpha| \sinh r \ge 1. \end{cases}
$$
In this case, we estimate (\ref{pfz}) by
$$
 \left|\psi  f^{\frac12} \zeta^{-\frac12} \right| \le  \begin{cases}
C_1(|\alpha| + e^{-r})^{\frac13} e^{r} + C_2e^{-r} (|\alpha|+ e^{-r})^{-\frac53}  & \text{for } |\alpha| \sinh r \le 1,\\
C_1 (1+|\alpha| e^r)^{-3} |\alpha|^{\frac13} (r + \log |\alpha|)^\frac13 e^r
+ C_2 |\alpha|^{-\frac52} r^{-\frac53}e^{-r} & \text{for } |\alpha| \sinh r \ge 1.
\end{cases}
$$
It is then straightforward to bound, for $|\alpha| \le 1$,
\begin{equation}\label{phi.bound2}
\int_{|\omega^2 + \alpha^2 \cosh^2(r)|\ge c} \left|\psi  f^{\frac12} \zeta^{-\frac12} \right| dr = O(|\alpha|^{-\frac23}).
\end{equation}

\emph{Case 3.}  Assume $|\omega^2 + \alpha^2 \cosh^2(r)|\le c$.  Here we are
near the turning point, where $\phi$ and $\zeta$ are small.  Since $|\omega^2 + \alpha^2 \cosh^2(r)|\le c$ 
implies $|\alpha|^2 \le  \omega^2 + c$, we are only concerned with small $|\alpha|$ here.
We proceed as in \cite[Appendix]{Borthwick:2009}.  In the coordinate $z = \sinh r$, the turning point occurs at
$$
z_0 = \sqrt{- 1 - \frac{\omega^2}{\alpha^2}}.
$$
Set 
\begin{equation}\label{p.def}
p(z) := \left( \frac{f}{z - z_0} \right)^{\frac12} = \frac{\alpha \sqrt{z+z_0}}{\sqrt{z^2 + 1}}.
\end{equation}
Because $|\omega^2 + \alpha^2 \cosh^2(r)| = |\alpha^2(z^2-z_0^2)|$, 
the assumption $|\omega^2 + \alpha^2 \cosh^2(r)|\le c$ implies
\begin{equation}\label{zz.a}
z \asymp z_0 \asymp |\alpha|^{-1},
\end{equation}
with constants that depend only on $c$.
This makes it easy to estimate
\begin{equation}\label{p.bounds}
|\del_z^k p(z)| \asymp |\alpha|^{\frac32+k},
\end{equation}
with constants that depend only on $c$ and $k$.  If we define 
$$
q(z) := \frac{\phi}{(z-z_0)^{\frac32}},
$$
then by writing
$$
q(z) = \int_0^1 t^{\frac12} \frac{p(z_0 + t(z-z_0))}{\sqrt{((1-t)z_0+tz)^2+1}} dt,
$$
we can deduce from (\ref{p.bounds}) that
\begin{equation}\label{q.bounds}
|\del_z^k q(z)| \asymp |\alpha|^{\frac52+k}.
\end{equation}

To apply these estimates, we note that $f/\zeta = p^2(\tfrac32 q)^{-\frac23}$.
We can use this identification to apply the bounds (\ref{p.bounds}) and (\ref{q.bounds})
to the formula (\ref{zz.a}) for $\psi$, obtaining
$$
\left|\psi  f^{\frac12} \zeta^{-\frac12} \right| 
\asymp |\alpha|^{-\frac23}\quad\text{for }|\omega^2 + \alpha^2 \cosh^2(r)|\le c.
$$
The bound,
\begin{equation}\label{phi.bound3}
\int_{|\omega^2 + \alpha^2 \cosh^2(r)|\le c} \left|\psi  f^{\frac12} \zeta^{-\frac12} \right| dr = O(|\alpha|^{-\frac23}),
\end{equation}
follows immediately, since the range of integration for $r$ is $O(1)$.  

Combining the bounds (\ref{phi.bound1}), (\ref{phi.bound2}), and (\ref{phi.bound3})
gives
$$ 
\Phi(0) = O(|\alpha|^{-\frac23}),
$$
and the claimed estimate follows from (\ref{h.bound}).
\end{proof}

\section{Funnel determinant estimates}\label{fdet.sec}

For the model funnel $F_\ell$, fix $r_0 \ge 0$ and for some $\eta >0$ set 
$$
r_k = r_0 + k \eta.
$$ 
Let $\chr_k$ denote the multiplication operator for the characteristic function of the
interval $r \in [r_k,r_{k+1}]$
in $L^2(F_\ell)$.  The operator $G_j(s)$ defined in (\ref{Gj.funnel}) can 
be represented in the model funnel case by
\begin{equation}\label{GFell.def}
G(s) :=  (2s-1) \chr_{1} E_{F_\ell}(1-s) E_{F_\ell}(s)^t  \chr_{2}
\end{equation}
Our goal in this section is to prove the sharp bound on $\log \det (1 + c|G(s)|)$ used in the 
proof of Theorem~\ref{tau.int.thm}.

To proceed we must analyze the Fourier decomposition of $E_{F_\ell}(s)$.
Becuse of the circular symmetry, the Poisson kernel on $F_\ell$
admits a diagonal expansion into Fourier modes:
\begin{equation}\label{EF.modes}
E_{F_\ell}(s;r, \theta, \theta') = \frac{1}{\ell}\sum_{k \in \bbZ} a_k(s; r) e^{ik(\theta - \theta')}
\end{equation}
The coefficients $a_k(s;r)$ satisfy (\ref{kmode.eq}) with the boundary condition $a_k(s; 0) =0$, 
so we must have 
\begin{equation}\label{ak.def}
a_k(s; r) = c_k(s) w^-_k(s; r),
\end{equation}
where $w^-_k$ is the odd solution (\ref{wm.def}).
To compute the normalization constant $c_k(s)$, we use the fact that
\begin{equation}\label{coeff.rho}
(2s-1) a_k(s; r) \sim \rho^{1-s} + [S_{F_\ell}(s)]_k(s) \rho^s,
\end{equation}
as $\rho \to 0$, where $[S_{F_\ell}(s)]_k(s)$ is the $k$-th matrix element of the scattering matrix
$S_{\ell}(s)$.  Applying the appropriate Kummer identity \cite[eq.~(5.10.16)]{Olver}
to the hypergeometric function in (\ref{wm.def}) gives
$$
a_k(s; r) \sim c_k(s) \Bigl[ \Gamma(\tfrac12 - s) \beta_k(2-s) \rho^s
+ \Gamma(s-\tfrac12) \beta_k(1+s) \rho^{1-s}\Bigr],
$$
where
\begin{equation}\label{beta.def}
\beta_k(s) := \frac{1}{\Gamma(\frac{s+i k\omega)}2) \Gamma(\frac{s-i k\omega)}2)}.
\end{equation}
By comparing this asymptotic to (\ref{coeff.rho}), we can read off the coefficient,
$$
c_k(s) = \frac{2s-1}{ 
\Gamma(s-\frac12) \beta_k(1+s)},
$$
as well as the scattering matrix element,
\begin{equation}\label{gk.def} 
[S_{F_\ell}(s)]_k(s) = \frac{\Gamma(\frac12 - s) \beta_k(2-s)}{\Gamma(s-\frac12) \beta_k(1+s)}.
\end{equation}
For future reference we note also that 
\begin{equation}\label{ak.inv}
a_k(1-s; r) = - \frac{a_k(s; r)}{[S_{F_\ell}(s)]_k(s)}.
\end{equation}
and
\begin{equation}\label{ak.sym}
a_k(s;r) = a_{-k}(s;r)
\end{equation}

We can express the singular
values of $G(s)$ in terms of the coefficients $a_k(s;r)$.  Up to reordering, 
these singular values are given by 
\begin{equation}\label{laml.def}
\lambda_k(s) := |2s-1| \>\left[ \int_{r_1}^{r_2} |a_k(1-s; r)|^2\>\cosh r\>dr \right]^{\frac12} 
\left[ \int_{r_2}^{r_3} |a_k(s; r)|^2\>\cosh r\>dr\right]^{\frac12},
\end{equation}
for $k \in \bbZ$.   To prove this, we note that $\lambda_k(s)^2$ is the eigenvalue
of $G^*G(s)$ corresponding to the eigenfunction $\chi_{[r_2, r_3]}(r) \overline{a_k(s;r)} e^{-ik\theta}$.
And it is easy to see from (\ref{GFell.def}) and (\ref{EF.modes}) that these are the only non-zero
eigenvalues.

Using (\ref{ak.inv}) to replace $a_k(1-s)$ by $a_k(s)$, and assuming $\eta \le 1$,
we can estimate
\begin{equation}\label{lk.est1}
\lambda_k(\tfrac12 + k\alpha) \le  \Bigl| 2k\alpha \>a_k(\tfrac12+k\alpha; r_3)^2\> 
\bigl[S_{F_\ell}(\tfrac12 - k\alpha)\bigr]_k\>\cosh r_3 \Bigr|.
\end{equation}
We will first estimate the various components.  Recall that the matrix elements of $S_{F_\ell}(s)$
were expressed in terms of the function $\beta_k$ defined in (\ref{beta.def}).
\begin{lemma}\label{gb.lemma}
For $k>0$ and $\arg \alpha \in [0, \tfrac{\pi}2]$, if we assume $\dist(k\alpha, \bbN_0) \ge \delta$ then we have
$$
\log \Bigl|\bigl[S_{F_\ell}(\tfrac12 - k\alpha)\bigr]_k \Bigr| \ge 2k\re \gamma + 2k[\re \phi_0]_-
- C(\delta)
$$
where $\gamma(\alpha)$ was defined in (\ref{gamma.def}).
If instead we assume that $\dist( \tfrac12 - k\alpha, \calR_{F_\ell}) \le |k\alpha|^{-\beta}$,
then
$$
\log \Bigl|\bigl[S_{F_\ell}(\tfrac12 - k\alpha)\bigr]_k \Bigr| \le 2k\re \gamma + 2k[\re \phi_0]_-
+ C(\beta) \log |k\alpha|.
$$
\end{lemma}
\begin{proof}
Consider the matrix element (\ref{gk.def}).
For $\re \alpha \ge 0$, we can apply Stirling's formula directly to obtain
$$
\log \Gamma(k\alpha) \beta_k(\tfrac32 + k\alpha) = k\gamma(\alpha) - 
\tfrac12 \log \pi k^2 \alpha \sqrt{\omega^2 + \alpha^2} +  O(|k\alpha|^{-1}),
$$
To estimate the other term, we must avoid zeros and poles.  For $\re z \le 0$,
applying Stirling via the reflection formula gives
$$
\log |\Gamma(z)| \le \re\bigl[(z-\tfrac12) \log (-z) - z\bigr] - \pi |\im z| + \log \bigl[1+ \dist(z, -\bbN_0)^{-1}\bigr] + O(1),
$$
and
$$
\log |\Gamma(z)| \ge \re\bigl[(z-\tfrac12) \log (-z) - z\bigr] - \pi |\im z| + O(1).
$$
If we assume that $\dist(k\alpha, \bbN_0) \ge \delta$, then we obtain the upper bound
$$
\log |\Gamma(-k\alpha) \beta_k(\tfrac32 - k\alpha)| \le - k\re \gamma(\alpha) - 2k [\re \phi_0]_-
- \tfrac12 \log k^2 \alpha \sqrt{\omega^2 + \alpha^2} + C(\delta).
$$
For a lower bound, we need to assume that $\dist(k\alpha, \calR_{F_\ell}) \ge |k\alpha|^{-\beta}$,
and then we find that
$$
\log |\Gamma(-k\alpha) \beta_k(\tfrac32 - k\alpha)| \ge - k\re \gamma(\alpha) - 2k [\re \phi_0]_-
- \tfrac12 \log k^2 \alpha \sqrt{\omega^2 + \alpha^2} - C(\beta) \log |k\alpha|.
$$
\end{proof}

\begin{lemma}\label{lk.lemma}
Assuming that $\re \alpha \ge 0$, $k>0$, 
$\dist( \tfrac12 - k\alpha, \calR_{F_\ell}) \le |k\alpha|^{-\beta}$, 
we have
$$
\log \lambda_k(\tfrac12+k\alpha) \le  2k \re\phi(\alpha; r_3)-2k[\re \phi_0(\alpha) ]_+ + O(\log |k\alpha|).
$$
\end{lemma}
\begin{proof}
By conjugation we can assume $\arg\alpha \in [0,\tfrac{\pi}2]$.
Then $a_k(\tfrac12+k\alpha; r)$ can be expressed in terms of the solutions $w_\sigma$
from Proposition~\ref{wsig.prop}.  To satisfy the Dirichlet boundary condition,
it must be a $w_0(0) w_1(r) - w_1(0) w_0(r)$.  Lemma~\ref{wsig.rlim} gives
the asymptotic behavior of this expression as $r \to \infty$.  After 
comparing to (\ref{coeff.rho}), we find that
\begin{equation}\label{ck.def}
a_k(\tfrac12+k\alpha; r) = \frac{1}{2k w_0(0)} \alpha^{-\frac12} 
e^{-k(\phi_0(\alpha)+\gamma(\alpha))} \bigl[ w_0(0) w_1(r) - w_1(0) w_0(r) \bigr]
\end{equation}
The estimate 
\begin{equation}\label{ak.bound}
|a_k(\tfrac12 + k\alpha; r)| \le C k^{\frac16}
e^{k\re[\phi(\alpha, r)-\phi_0(\alpha)-\gamma(\alpha)]},
\end{equation}
for $|k\alpha|$ sufficiently large,
then follows immediately from (\ref{wsig.upper}) and (\ref{w0.lower}).
The result now follows from applying Lemma~\ref{gb.lemma} and (\ref{ak.bound}) in (\ref{lk.est1}).
\end{proof}

\bigbreak
\begin{proposition}\label{detG.prop}
Assuming that $\eta \le 1$, $0 \le \theta \le \tfrac{\pi}2$, 
and $\dist(\frac12 - ae^{i\theta}, \calR_{F_\ell}) \ge a^{-\beta}$
for some fixed $\beta>1$, we have
$$
\log \det\Bigl( I + c\>\bigl|G(\tfrac12 + ae^{i\theta})\bigr|\Bigr) \le \kappa(\theta, r_4) a^2 + C(c, r_0, \beta) a \log a,
$$
where
\begin{equation}\label{b.def}
\kappa(\theta, r) = 2 \int_0^{\infty} \frac{[I( xe^{i\theta},\ell, r)]_+}{x^3}\>dx - \frac12 \ell \sin^2\theta,
\end{equation}
with $I(xe^{i\theta},\ell, r) := 2\re \phi(xe^{i\theta}; r)$, which agrees with the definition (\ref{Idef}).
\end{proposition}

\begin{proof}
We start from the expression for the determinant in terms of the singular values,
$$
\det(I + c\>|G(\tfrac12 + ae^{i\theta})|) = \prod_{k\in\bbZ} (1+ c\lambda_k(\tfrac12 + ae^{i\theta})).
$$
By the conjugation symmetry, we can assume $\theta \in [0, \tfrac{\pi}2]$.
Let $\varrho(\theta)$ be the implicit solution of the equation $\re\phi(\varrho(\theta) e^{i\theta}, r_3) = 0$,
as shown in Figure~\ref{phicontour}.
\begin{figure} 
\psfrag{iw}{$i\omega$}
\psfrag{iwcosh2r}{$\frac{i\omega}{\cosh^2 r}$}
\psfrag{rp0}{$\re\phi_0 > 0$}
\psfrag{rpp}{$\re\phi >0 >\re\phi_0$}
\psfrag{rpm}{$\re\phi <0$}
\psfrag{arho}{$\alpha = \varrho(\theta) e^{i\theta}$}
\begin{center}  
\includegraphics{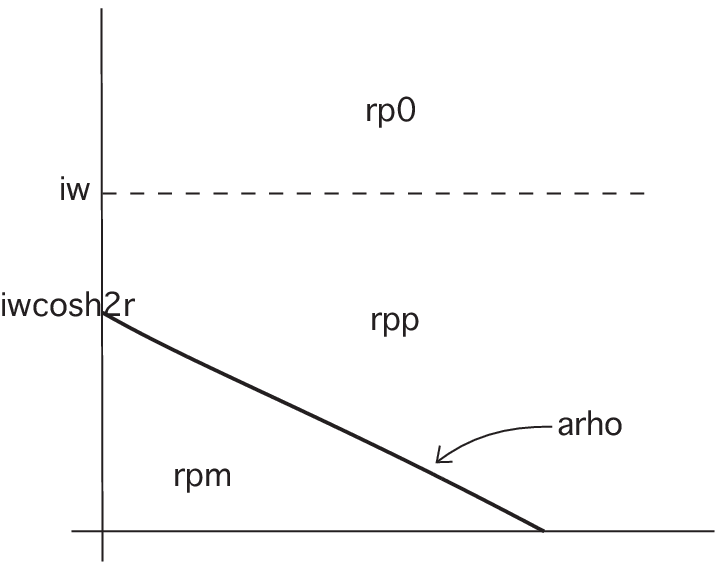} 
\end{center}
\caption{Positive and negative regions for $\re\phi(\alpha; r)$ and $\re\phi_0(\alpha)$, shown
for $r = 1$.}\label{phicontour}
\end{figure}
Note that $\re \phi_0(xe^{i\theta}) = 0$ in a neighborhood of $x = \varrho(\theta)$.
For some $\delta>0$, we subdivide the sum in
$$
\log \det\Bigl( I + c\>\bigl|G(\tfrac12 + a e^{i\theta})\bigr|\Bigr) 
= 2\sum_{k=1}^\infty \log ( 1 + c\lambda_k(\tfrac12 + a_ie^{i\theta})) + O(a\log a)
$$
at values where $a_i/k = \varrho(\theta)$ and $(1-\delta)\varrho(\theta)$.  
The dominant part of the sum is
$$
\Sigma_+ := \sum_{1 \le k \le a/\varrho(\theta)} \log ( 1 + c\lambda_k(\tfrac12 + ae^{i\theta})).
$$
Assuming that $a \in \{a_i\}$, Lemma~\ref{lk.lemma} gives the bound
$$
\Sigma_+ \le \sum_{1 \le k \le a/\varrho(\theta)} 2k \Bigl( \re \phi \bigl( ae^{i\theta}/k; r_3\bigr) 
- \bigl[ \re \phi_0 \bigl( ae^{i\theta}/k \bigr)\bigr]_+ \Bigr) + C(c, r_0, \beta) a\log a.
$$
Because the summand is a decreasing function of $k$, we may estimate the sum by the integral
$$
\Sigma_+ \le  \int_0^{\frac{a}{\varrho(\theta)}} 2k \Bigl( \re \phi \bigl( ae^{i\theta}/k; r_3\bigr) 
- \bigl[ \re \phi_0 \bigl( ae^{i\theta}/k \bigr)\bigr]_+ \Bigr)  + C(c, r_0, \beta) a\log a
$$
Substituting $x = a/k$ gives
$$
\int_0^{\frac{a}{\varrho(\theta)}} 2k \re \phi \bigl( ae^{i\theta}/k; r_3\bigr) \>dk
=  2a^2 \int_{\varrho(\theta)}^{\infty} \frac{\re\phi(xe^{i\theta}; r_3)}{x^3}\>dx.
$$
We can also compute that
\[
\begin{split}
\int_0^{\frac{a}{\varrho(\theta)}} 2k \bigl[ \re \phi_0 \bigl( ae^{i\theta}/k \bigr)\bigr]_+ \>dk
& =  \pi a^2 \int_{\omega/\sin(\theta)}^{\infty} \frac{x\sin \theta - \omega}{x^3}\>dx, \\
& = \frac{\pi a^2}{2\omega} \sin\theta.
\end{split}
\]
Comparing to (\ref{b.def}), we conclude that
$$
\Sigma_+ \le \kappa(\theta, r_3)a^2 + C(c, r_0, \beta) a \log a.
$$

The middle term is given by
$$
\Sigma_0 := \sum_{a/\varrho(\theta) \le k \le a/(1-\delta)\varrho(\theta)} \log ( 1 + c\lambda_k(\tfrac12 + ae^{i\theta})),
$$
Since $I(\alpha,\ell, r_3) = O(\delta)$ for $k$ in this range, the same integral estimate used for $\Sigma_+$
gives
$$
|\Sigma_0| \le C(c, r_0, \beta) \delta a^2 + C(c, r_0, \beta) a \log a.
$$

Finally, we set
$$
\Sigma_- := \sum_{k \ge a/(1-\delta)\varrho(\theta)}  \log ( 1 + c\lambda_k(\tfrac12 + ae^{i\theta})).
$$
For $k$ in this range, $I(\alpha,\ell, r_3) \le - C \delta$ and we can estimate
$$
|\Sigma_-| \le C(c, r_0, \beta, \delta) e^{-ca},
$$
for some $c>0$.

Adding together the estimates for $\Sigma_+$, $\Sigma_0$, and $\Sigma_-$ gives
$$
\log \det\Bigl( I + C\>\bigl|G(\tfrac12 + a e^{i\theta})\bigr|\Bigr) 
\le \kappa(\theta, r_3) a^2 + C(c, r_0, \beta) [\delta a^2 + a \log a] + C(c, r_0, \beta, \delta) e^{-ca}
$$
We can absorb the $\delta a^2$ term into the first term by replacing $r_3$ by $r_4$, 
assuming that $\eta = O(\delta)$, since $\kappa(\theta, \cdot)$ is strictly increasing.  
This yields the claimed estimate.
\end{proof}

\section{Resonance asymptotics for truncated funnels}\label{trfun.sec}

Inside the model funnel $F_\ell$, with metric given by (\ref{fun.metric}), we let  
$F_{\ell, r_0}$ denote the truncated region $\{r \ge r_0\}$, with the Laplacian 
defined by imposing Dirichlet boundary conditions at $r = r_0$.  To compute 
the associated scattering matrix elements exactly, we consider the solutions of the Fourier 
mode equation (\ref{kmode.eq}) given by (\ref{wp.def}) and (\ref{wm.def}).  
To impose the boundary condition at $r=r_0$, we set
\begin{equation}\label{uk.tr}
u_k(s;r) := w^+_k(s; r_0) w^-_k(s;r) -  w^-_k(s; r_0) w^+_k(s;r).
\end{equation}
The scattering matrix element may be obtained from the asymptotics of $u_k(s;r)$ as 
$r\to \infty$ be noting that for any generalized eigenmode we have
\begin{equation}\label{uk.asymp}
u_k(s;r) \sim c_{k,s} \Bigl( \rho^{1-s} + [S_{F_{\ell, r_0}}(s)]_k \rho^s \Bigr),
\end{equation}
as $r \to \infty$, where $\rho := 2e^{-r}$ as before.
The solutions $w^\pm_k$ have leading asymptotics,
\begin{equation}\label{wpm.rinf}
\begin{split}
w^+_k(s; r) &\sim \Gamma(s-\tfrac12) \beta_k(s) \rho^{1-s} + \Gamma(\tfrac12-s) \beta_k(1-s) \rho^{s}, \\
w^-_k(s; r) &\sim \Gamma(s-\tfrac12) \beta_k(1+s) \rho^{1-s} + \Gamma(\tfrac12-s) \beta_k(2-s) \rho^{s},
\end{split}
\end{equation}
as $r\to \infty$, where $\beta_k(s)$ was defined in (\ref{beta.def}). 

If we set
\begin{equation}\label{fk.def}
f_k(s; r) :=  \Gamma(s-\tfrac12) \Bigl[ \beta_k(1+s)  w^+_k(s;r) - \beta_k(s) w^-_k(s;r) \Bigr],
\end{equation} 
Then from (\ref{uk.asymp}) we can read off that
\begin{equation}\label{SFlr}
\bigl[S_{F_{\ell, r_0}}(s)\bigr]_k =  \frac{f_k(1-s; r_0)}{f_k(s; r_0)}.
\end{equation}
The $k$-th Fourier mode thus contributes scattering poles at the values of $s$ for which
$$
\beta_k(1+s)  w^+_k(s;r_0) - \beta_k(s) w^-_k(s;r_0) = 0.
$$  
This function can be written in terms of a single 
normalized hypergeometric function, via the standard identities, yielding
$$
\calR_{F_{\ell, r_0}} = \bigcup_{k\in \bbZ} \Bigl\{s:\> \bF(\tfrac{1+s+i\omega k}{2}, \tfrac{s+i\omega k}{2};
\tfrac12+s; -\sinh^{-2}r_0) = 0 \Bigr\}.
$$
A sample resonance counting function is shown in Figure~\ref{Nplot}.

\begin{theorem}\label{trfun.thm}
For the truncated funnel with Dirichlet boundary conditions,
$$
N_{F_{\ell.r_0}}(t) \sim A(F_{\ell.r_0}) t^2,
$$
where $A(F_{\ell.r_0})$ is given by (\ref{A.f}).
\end{theorem}
In conjunction with \cite[Thm.~1.2]{Borthwick:2009} for the hyperbolic planar case, this will complete the proof of
Theorem~\ref{Yj.asymp}.  Before giving the proof, we need some estimates of scattering matrix elements.

\begin{figure} 
\psfrag{N}{$N_{F_{\ell, r_0}}(t)$}
\psfrag{2}{$2$}
\psfrag{4}{$4$}
\psfrag{6}{$6$}
\psfrag{8}{$8$}
\psfrag{10}{$10$}
\psfrag{100}{$100$}
\psfrag{200}{$200$}
\psfrag{300}{$300$}
\begin{center}  
\includegraphics{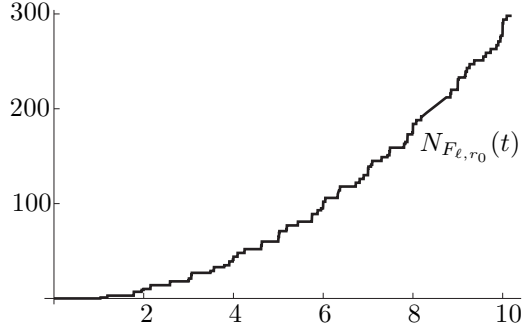} 
\end{center}
\caption{The resonance counting function for $F_{\ell, r_0}$, shown for $\ell = 2\pi$ and $r_0 = 1$.}\label{Nplot}
\end{figure}

\begin{lemma}\label{SFlr.est}
Assuming that $\arg \alpha \in [0, \tfrac{\pi}2-\vep]$ with $\dist(k\alpha, \bbN_0) 
\ge \eta$, we can have
$$
\log \left| \frac{[S_{F_{\ell, r_0}}(\tfrac12+k\alpha)]_k}{[S_{F_{\ell}}(\tfrac12+k\alpha)]_k} - 1 \right| 
\ge 2k \Bigl(\re\phi(\alpha; r_0) - [\re \phi_0(\alpha)]_+\Bigr) - C(\eta), 
$$
for $|k\alpha|$ sufficiently large.
\end{lemma}

\begin{proof}
To estimate $[S_{F_{\ell, r_0}}(s)]_k$, as given in (\ref{SFlr}), we must connect $f_k$
to the solutions $w_\sigma$ introduced in (\ref{wsig.def}).  
Since $f_k(\tfrac12+k\alpha;r)$ is recessive as $r \to \infty$, this solution must be proportional to $w_0$.
From (\ref{wpm.rinf}), we can use the reflection formula for the Gamma function to see that
$$
f_k(\tfrac12+k\alpha; r) \sim \frac{\rho^s}{\pi k \alpha},
$$
as $r \to \infty$.
By comparing this to the asymptotic from Lemma~\ref{wsig.rlim}, we can see that
\begin{equation}\label{fk.w0}
f_k(\tfrac12+k\alpha; r)  = A^+_0 w_0(r).
\end{equation}
where
$$
A^+_0 :=   \frac{1}{\pi k\sqrt{\alpha}} e^{k(\phi_0 + \gamma)}.
$$
We may also $f_k(\tfrac12 - k\alpha; r)$ in terms of the $w_\sigma$,
\begin{equation}\label{fk.Aw}
f_k(\tfrac12 - k\alpha; r) = A^-_0w_0(r) + A^-_1w_1(r),
\end{equation}
for some $A^-_0$, $A^-_1$ that are independent of $r$ but do depend on $k$ and $\alpha$.
By (\ref{wpm.rinf}),
$$
f_k(\tfrac12-k\alpha; r) \sim - \frac{\rho^{1-s}}{\pi k\alpha},
$$
and so by Lemma~\ref{wsig.rlim} we have
\begin{equation}\label{A1.def}
A^-_1 =  - \pi^{-1}k^{-1} \alpha^{-\frac12} e^{-k(\phi_0 + \gamma)}.
\end{equation}
The other coefficient can then be computed by comparing values at $r=0$,
\begin{equation}\label{A0.def}
A^-_0 =  \frac{1}{w_0(0)} \left[ f_k(\tfrac12-k\alpha; 0) - A^-_1 w_1(0)\right].
\end{equation}

Using (\ref{fk.w0}) to relate $w_0(0)$ to $f_k(\tfrac12+k\alpha; 0)$, we can then deduce that
\begin{equation}\label{Str.Sw}
\bigl[S_{F_{\ell, r}}(\tfrac12+k\alpha)\bigr]_k  =  \bigl[S_{F_{\ell}}(\tfrac12+k\alpha)\bigr]_k
 - e^{-2k(\phi_0+\gamma)}\left( \frac{w_1(r)}{w_0(r)}
- \frac{w_1(0)}{w_0(0)} \right).
\end{equation}
Hence
\begin{equation}\label{Str.w01}
\frac{[S_{F_{\ell, r}}(\tfrac12+k\alpha)]_k}{[S_{F_{\ell}}(\tfrac12+k\alpha)]_k} - 1  =  
- e^{-2k(\phi_0+\gamma)}\left( \frac{w_1(r)}{w_0(r)} - \frac{w_1(0)}{w_0(0)} \right) 
[S_{F_{\ell}}(\tfrac12-k\alpha)]_k
\end{equation}

For $\arg \alpha \in [0, \tfrac{\pi}2 - \vep]$, we deduce from (\ref{wsig.asymp})
(using also the fact that $\re(\phi-\phi_0)> c(\vep, r))$ that
\begin{equation}\label{w01.est}
\left( \frac{w_1(r)}{w_0(r)} - \frac{w_1(0)}{w_0(0)} \right) = e^{2k\phi} (1 + O(|k\alpha|^{-1})).
\end{equation}
The result then follows from (\ref{Str.w01}) and the lower bound on $[S_{F_{\ell}}(\tfrac12-k\alpha)]_k$
provided by Lemma~\ref{gb.lemma}. 
\end{proof}

The estimates in Lemma~\ref{SFlr.est} give approximate locations for the resonances
in $\calR_{F_{\ell, r_0}}$ arising from the $k$-th Fourier mode.  
The zeros of (\ref{Str.Sw}) correspond to resonances at 
$s = \tfrac12 - k\alpha$.   This requires a cancellation between the two terms on the
right-hand side of (\ref{Str.Sw}).  If $\re \phi >0$, then the second term is larger 
by approximately $e^{2k \phi}$ and cancellation only occurs near the poles of $[S_{F_\ell}(s)]_k$;
this explains the poles of $[S_{F_{\ell, r_0}}(s)]_k$ on the negative real axis.  For $\re \phi = 0$ the two
terms in (\ref{Str.Sw}) have the same magnitiude;  the resonances off the real axis in $\calR_{F_{\ell, r_0}}$
thus occur near the line $\re \phi((\tfrac12 - s)/k; r_0) = 0$ (and its conjugate).   Figure~\ref{resphi7}
illustrates this phenomenon.   For $\re \phi < 0$ the first term in (\ref{Str.Sw}) is always larger than
the second and no zeros occur.
\begin{figure} 
\psfrag{4}{$4$}
\psfrag{-10}{$-10$}
\psfrag{rph}{$\re\phi((\tfrac12-s)/k; r_0) = 0$}
\begin{center}  
\includegraphics{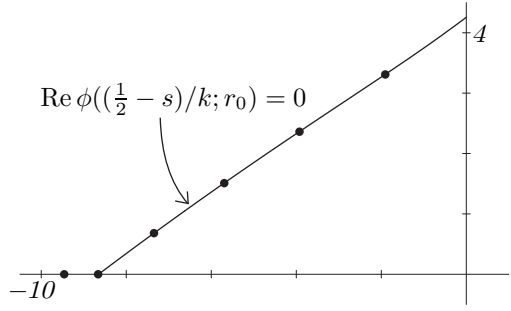} 
\end{center}
\caption{Using the equation $\re \phi = 0$ to locate the resonances of 
$F_{\ell,r_0}$ occurring in the $k=7$ Fourier mode, shown for $\ell = 2\pi$ and $r_0 = 1$.}\label{resphi7}
\end{figure}

Since $[S_{F_{\ell, r}}(\tfrac12+k\alpha)]_k$ may indeed have zeros near the line
$\re \phi = 0$, proving a lower bound is more delicate in this region.  By focusing on
a relatively narrow strip, we can settle for a cruder estimate
on the matrix elements in the vicinity of the zeros.
\begin{lemma}\label{SFlr.mm}
For $k \ge 0$ and $\re s \ge \tfrac12$ and assuming $\dist(1-s, \calR_{F_{\ell}}) \ge |s|^{-\beta}$ with $\beta >2$,
$$
\log \left| \frac{\bigl[S_{F_{\ell, r_0}}(s)\bigr]_k}{\bigl[S_{F_{\ell}}(s)\bigr]_k} \right| 
\le  C(r_0,\beta) (k + |s|) \log |s|.
$$
If $\dist(1-s, \calR_{F_{\ell,r_0}}) \ge |s|^{-\beta}$ with $\beta >2$, then we have
$$
\log \left| \frac{\bigl[S_{F_{\ell, r_0}}(s)\bigr]_k}{\bigl[S_{F_{\ell}}(s)\bigr]_k} \right| 
\ge - c(r_0,\beta) (k + |s|) \log |s|.
$$
\end{lemma}
\begin{proof}
From (\ref{fk.def}), we note that $f_k(s;r_0)/\Gamma(s-\tfrac12)$
is an entire function of $s$.  By Stirling's formula and the estimate (\ref{wsig.upper}), we can
estimate its growth for large $|s|$ and $k \ne 0$ by
$$
\log \left| \frac{f_k(s;r_0)}{\Gamma(s-\tfrac12)} \right| \le C(r_0)(k + |s|) \log |s|),
$$
where $C$ is independent of $k$.  The same estimate holds for $k=0$, by the classical
asymptotics of the hypergeometric function due to Watson \cite[\S2.3.2]{Erdelyi}.
Assuming that $\dist(s, \calR_{F_{\ell,r_0}}) \ge |s|^{-\beta}$,
where $\beta>2$, the Minimum Modulus Theorem (see e.g.~\cite[Thm.~3.7.4]{Boas})
gives
$$
\log \left| \frac{f_k(s;r_0)}{\Gamma(s-\tfrac12)} \right| \ge -c(r_0,\beta) (k + |s|) \log |s|,
$$
for large $|s|$.  The results follow from applying these estimates to 
$$
\frac{\bigl[S_{F_{\ell, r_0}}(s)\bigr]_k}{\bigl[S_{F_{\ell}}(s)\bigr]_k} =
\frac{f_k(1-s; r_0)}{f_k(s; r_0)} \frac{f_k(s; 0)}{f_k(1-s; 0)}.
$$
\end{proof}

\bigbreak
\begin{proof}[Proof of Theorem~\ref{trfun.thm}]
We note that
$$
N_{F_\ell}(t) \sim \frac{\ell}4 t^2,
$$
and
$$
\ovol(F_{\ell, r_0}) = - \ell \sinh r_0.
$$
By Corollary~\ref{relcount} and Theorem~\ref{tau.int.thm}, the claimed asymptotic will be proven if we can show
that there exists an unbounded set $\Lambda \subset [1, \infty)$ such that
\begin{equation}\label{tau.trf.asym}
\frac{2}{\pi} \int_{0}^{\frac{\pi}2} \log |\tau(\tfrac12+a e^{i\theta})| \>d\theta 
\ge  \frac{4a^2}{\pi} \int_{0}^{\frac{\pi}2} \int_0^\infty \frac{[I(xe^{i\theta},\ell,r_0)]_+}{x^3}\>dx
- \frac{\ell}4 a^2 - o(a^2),
\end{equation}
for all $a \in \Lambda$.  We take
\begin{equation}\label{Lambda.RRN}
\Lambda := \left\{ a \ge 1:\> \dist\Bigl(\bigl\{|s-\tfrac12| = a\bigr\}, \calR_{F_{\ell}}\cup \calR_{F_{\ell, r_0}} 
\cup \bbN_0 \Bigr) \ge a^{-3} \right\}.
\end{equation}
 
Using the symmetry of coefficients under $k \to -k$, and estimating the $k=0$ term by 
Lemma~\ref{SFlr.mm}, we have
\begin{equation}\label{lt.sum}
\log |\tau(\tfrac12+ae^{i\theta})| =  2\sum_{k=1}^\infty \log \left|  
\frac{[S_{F_{\ell, r_0}}(\tfrac12+ae^{i\theta})]_k}{[S_{F_\ell}(\tfrac12+ae^{i\theta})]_k} \right|
+ O(a\log a).
\end{equation}

Define $\varrho(\theta)$ by $\re \phi(\varrho(\theta) e^{i\theta}, r_0) = 0$, as in the proof
of Proposition~\ref{detG.prop}, and assume for now that $\theta \le \tfrac{\pi}2 - \vep$.  
For $\delta>0$, we will split the sum (\ref{lt.sum}) at 
$a/k = \varrho(\theta)(1\pm a^{-1/2})$.  Let $\Sigma_+$ denote the portion of the sum with 
$a/k \ge \varrho(\theta)(1+ a^{-1/2})$.  
Under this condition, we want to derive a lower bound from Lemma~\ref{SFlr.est} 
using the inequality,
$$
\log |1 + \lambda| \ge \log |\lambda| - \log 2, \qquad\text{for }|\lambda| \ge 2.
$$
For $a$ sufficiently large, we will have 
$\re\phi(xe^{i\theta}, r_0) \ge ca^{-1/2}$ for $x \ge \varrho(\theta)(1 + a^{-1/2})$.
Thus, for $k \ge c\sqrt{a}$ we can deduce from Lemma~\ref{SFlr.est} that
$$
\log \left| \frac{[S_{F_{\ell, r_0}}(\tfrac12+ae^{i\theta})]_k}{[S_{F_\ell}(\tfrac12+ae^{i\theta})]_k} \right|
\ge 2k \Bigl( \re \phi \bigl( ae^{i\theta}/k; r_3\bigr) 
- \bigl[ \re \phi_0 \bigl( ae^{i\theta}/k \bigr)\bigr]_+ \Bigr) + O(1).
$$
Arguing as in the proof of Proposition~\ref{detG.prop}, we can then obtain
\[
\begin{split}
& \sum_{c\sqrt{a} \le k \le \frac{a}{\varrho(\theta)(1+a^{-1/2})}} \log 
\left|\frac{[S_{F_{\ell, r_0}}(\tfrac12+ae^{i\theta})]_k}{[S_{F_\ell}(\tfrac12+ae^{i\theta})]_k} \right| \\
&\qquad \ge  2a^2 \int_{\varrho(\theta)(1+a^{-1/2})}^{C\sqrt{a}}  
\frac{\re\phi(x e^{i\theta},r_0) - [\re\phi_0(x e^{i\theta})]_+}{x^3}\>dx   - O(a\log a).
\end{split}
\]
For $k \le c\sqrt{a}$, Lemma~\ref{SFlr.mm} gives the estimate
$$
\sum_{1 \le k \le c\sqrt{a}} \log 
\left|\frac{[S_{F_{\ell, r_0}}(\tfrac12+ae^{i\theta})]_k}{[S_{F_\ell}(\tfrac12+ae^{i\theta})]_k} \right|
\ge - O(a^\frac32 \log a).
$$
On the other hand, since $|\re\phi(\alpha, r)| = O(|\alpha|)$ for large $|\alpha|$, we also have
$$
2a^2 \int_{C\sqrt{a}}^{\infty}  \frac{\re\phi(x e^{i\theta},r_0) - [\re\phi_0(x e^{i\theta})]_+}{x^3}\>dx   =
O(a^\frac32).
$$
We can also estimate
$$
2a^2 \int_{\varrho(\theta)}^{\varrho(\theta)(1+a^{-1/2})} \frac{\re\phi(x e^{i\theta},r_0) 
- [\re\phi_0(x e^{i\theta})]_+}{x^3}\>dx =  O(a^\frac32),
$$
since $\re\phi(\alpha, \ell, r_0)$ is $O(\delta)$ in the range of integration.
In combination, these estimates give
\begin{equation}\label{lt.est1}
\Sigma_+ \ge 
2a^2 \int_{\varrho(\theta)}^{\infty}  \frac{\re\phi(x e^{i\theta},r_0)}{x^3}\>dx - \frac{\pi a^2}{2\omega} \sin^2\theta
- O(a^\frac32 \log a),
\end{equation}
for $a\in \Lambda$.

Let $\Sigma_0$ denote the portion of the sum in (\ref{lt.sum}) for which
$\varrho(\theta)(1-a^{-1/2}) < a/k < \varrho(\theta)(1+a^{-1/2})$.  Since there are $O(a^{1/2})$ values
of $k$ in this range, Lemma~\ref{SFlr.mm} gives the estimate
\begin{equation}\label{lt.est2}
\Sigma_0 \ge - O(a^\frac32 \log a).
\end{equation}

Finally, we have $\Sigma_-$, defined as the portion of (\ref{lt.sum}) with $a/k \le \varrho(\theta)(1-a^{-1/2})$.
Now we wish to apply Lemma~\ref{SFlr.est} using
$$
\log |1 + \lambda| \ge - |\lambda| \log 4, \qquad\text{for }|\lambda| \le \tfrac12.
$$
Note that $I(xe^{i\theta}, \ell, r_0) \le - ca^{-1/2}$ for $x \le \varrho(\theta)(1- a^{-1/2})$ and
$a$ sufficiently large, and that $k \ge ca$ in the range of $\Sigma_-$.  Thus for
large $a$ Lemma~\ref{SFlr.est} yields
$$
\log \left| \frac{[S_{F_{\ell, r_0}}(\tfrac12+ae^{i\theta})]_k}{[S_{F_\ell}(\tfrac12+ae^{i\theta})]_k} \right|
\ge - O(e^{-cka^{-1/2}}),
$$
within the scope of $\Sigma_-$.  We conclude that
\begin{equation}\label{lt.est3}
\Sigma_- \ge - O(e^{-ca^{1/2}}).
\end{equation}

Applying the estimates (\ref{lt.est1}), (\ref{lt.est2}), and (\ref{lt.est3}) to the sum (\ref{lt.sum})
now proves the lower bound
\[
\begin{split}
\frac{2}{\pi} \int_{0}^{\frac{\pi}2-\vep} \log |\tau(\tfrac12+a e^{i\theta})| \>d\theta 
& \ge  \frac{4a^2}{\pi} \int_{0}^{\frac{\pi}2-\vep} \int_0^\infty 
\frac{[2\re\phi(xe^{i\theta}, r_0)]_+}{x^3}\>dx  \\
&\qquad - \frac{2a^2}{\omega} \int_{0}^{\frac{\pi}2-\vep} \sin^2\theta \>d\theta
- o(a^2),
\end{split}
\]
For the missing sectors, we appeal to Lemma~\ref{lind.lemma} to see that
$$
\frac{2}{\pi} \int_{\frac{\pi}2-\vep}^{\frac{\pi}2} \log |\tau(\tfrac12+a e^{i\theta})| \>d\theta 
\ge  -c \vep a^2.
$$
We can thus take $\vep\to 0$ to complete the proof of (\ref{tau.trf.asym}).
\end{proof}

\bigbreak
\emph{Remark.}  In the proof of (\ref{A.hp}) given in \cite[Thm.~1.2]{Borthwick:2009},
the $\Sigma_-$ term was estimated incorrectly.  This term is not necessarily positive,
so the upper bound $O(e^{-ca})$ does not imply a corresponding lower bound.  Instead,
one needs to argue as in the derivation of (\ref{lt.est3}) above.
The estimates needed for the correct argument were given in \cite[eq. (6.8--6.10)]{Borthwick:2009}.

\section{Resonance asymptotics for extended funnels}\label{exfun.sec}

Using the same notation as in \S\ref{trfun.sec}, we now consider $F_{\ell, -r_0}$, defined
as the subset $r \ge -r_0$ in a hyperbolic cylinder of diameter $\ell$, where $r_0 \ge 0$.  
The metric and Laplacian are still given by (\ref{fun.metric}) and (\ref{fun.lapl}), 
so that the scattering matrix elements are easily computed in terms of hypergeometric functions
as before.

With reference to the even/odd solutions $w_k^\pm$ defined in (\ref{wp.def}) and (\ref{wm.def}),
a solution $u_k(s;r)$ to the $k$-th eigenmode equation (\ref{kmode.eq}) satisfying
$u_k(s; -r_0) = 0$ can be written
$$
u_k(s; r) = w^+_k(s;r_0) w^-_k(s; r) + w^-_k(s;r_0) w^+(s;r),
$$
where $w^\pm_k(s; r)$ are the even/odd hypergeometric solutions defined in (\ref{wp.def})
and (\ref{wm.def}).  Using the asymptotic expansions (\ref{wpm.rinf}) as $r \to \infty$, 
we can read off the scattering matrix elements
\begin{equation}\label{S.exfun}
[S_{F_{\ell, -r_0}}(s)]_k =  \frac{\Gamma(\tfrac12-s)}{\Gamma(s-\tfrac12)}
\frac{\beta_k(2-s) w^+_k(s;r_0) + \beta_k(1-s) w^-_k(s;r_0)}{\beta_k(1+s) w^+_k(s;r_0) 
+ \beta_k(s) w^-_k(s;r_0)},
\end{equation}
where $\beta_k(s)$ was defined in (\ref{beta.def}).

This shows in particular that
$$
\calR_{F_{\ell, -r_0}} = \bigcup_{k\in\bbZ} \Bigl\{ s:\>  \beta_k(1+s) w^+_k(s;r_0) 
+ \beta_k(s) w^-_k(s;r_0)  = 0\Bigr\}.
$$

\begin{theorem}\label{exfun.thm}
For the extended funnel with Dirichlet boundary conditions imposed at $r = -r_0$, 
where $r_0 \ge 0$, we have
$$
N_{F_{\ell, -r_0}}(t) \sim A(F_{\ell, -r_0}) t^2,
$$
where
\begin{equation}\label{A.exf}
A(F_{\ell, -r_0}) =  \frac{\ell}{2\pi} \sinh r_0 + \frac{4}{\pi} \int_{0}^{\frac{\pi}2} \int_{0}^\infty 
\frac{[I(xe^{i\theta},\ell, -r_0)]_+}{x^3}\>dx\>d\theta,
\end{equation}
and $I(\alpha, \ell, r)$ was defined in (\ref{Idef}).
\end{theorem}

\begin{proof}
Since $N_{F_\ell}(t) \sim \tfrac{\ell}4 t^2$ and $\ovol(F_{\ell, -r_0}) = \ell \sin r_0$,
Theorem~\ref{exfun.thm} will follow from Corollary~\ref{relcount} and Theorem~\ref{scphase.thm},
once we establish 
\begin{equation}\label{extau.asym}
\frac{2}{\pi} \int_{0}^{\frac{\pi}2} \log |\tau(\tfrac12 + a e^{i\theta})|\>d\theta = 
\frac{4a^2}{\pi} \int_{0}^{\frac{\pi}2} \int_0^\infty 
\frac{[I(xe^{i\theta},\ell, -r_0)]_+}{x^3}\>dx\>d\theta
- \frac{\ell}4 a^2 - o(a^2),
\end{equation}
where $\Lambda$ is defined again by (\ref{Lambda.RRN}).

As in the proof of Theorem~\ref{trfun.thm}, we start with the Fourier decomposition of the scattering
matrices and use Lemma~\ref{SFlr.mm} to estimate the $k=0$ term, leaving
\begin{equation}\label{exlt.sum}
\log |\tau(\tfrac12+ae^{i\theta})| =  2\sum_{k=1}^\infty \log \left|  
\frac{[S_{F_{\ell, -r_0}}(\tfrac12+ae^{i\theta})]_k}{[S_{F_\ell}(\tfrac12+ae^{i\theta})]_k} \right|
+ O(a\log a).
\end{equation}
If we define
$$
g_k(s; r) := \Gamma(s-\tfrac12) \Bigl[\beta_k(1+s) w^+_k(s;r) + \beta_k(s) w^-_k(s;r)\Bigr],
$$
then by (\ref{S.exfun}), 
$$
[S_{F_{\ell, -r_0}}(\tfrac12+ae^{i\theta})]_k  = \frac{g_k(\tfrac12-ae^{i\theta})}{g_k(\tfrac12+ae^{i\theta})}
$$
Assuming $k > 0$, we set $k\alpha = ae^{i\theta}$.
Since $g_k(s; \cdot)$ solves (\ref{kmode.eq}), for $\re\alpha \ge 0$ we
can write
$$
g_k(\tfrac12\pm k\alpha; r) = B_0^\pm w_0(r) + B_1^\pm w_1(r),
$$
where $w_\sigma$ are the solutions given in (\ref{wsig.def}). 

As $r \to \infty$, the coefficient of $\rho^{1-s}$ in the expansion of $g_k(\tfrac12+k\alpha; r)$ is
\begin{equation}\label{gkp.asym}
2 \Gamma(k\alpha)^2 \beta_k(\tfrac12+k\alpha) \beta_k(\tfrac32+k\alpha) 
=  \frac{1}{\pi k\alpha} \left(1 - \frac{\cosh \pi k \omega}{\sin \pi k \alpha}\right) 
\bigl[S_{F_{\ell}}(\tfrac12-k\alpha)\bigr]_k.
\end{equation}
The coefficient of $\rho^{1-s}$ in $g_k(\tfrac12-k\alpha; r)$ is
\begin{equation}\label{gkm.asym}
\begin{split}
&\Gamma(k\alpha) \Gamma(-k\alpha) \left[\beta_k(\tfrac12+k\alpha) 
\beta_k(\tfrac32-k\alpha) + \beta_k(\tfrac12-k\alpha) \beta_k(\tfrac32+k\alpha) \right] \\
&\qquad =  - \frac{1}{\pi k\alpha} \frac{\cosh \pi k \omega}{\sin \pi k \alpha}.
\end{split}
\end{equation}
Comparing these to the asymptotics for $w_\sigma$, as given in Lemma~\ref{wsig.rlim},
we see that
\begin{equation}\label{B1p}
B_1^+ =  \frac{e^{-k(\phi_0+ \gamma)}}{\pi k\sqrt{\alpha}} \left(1 - \frac{\cosh \pi k \omega}{\sin \pi k \alpha}\right) 
 \bigl[S_{F_{\ell}}(\tfrac12-k\alpha)\bigr]_k,
\end{equation}
and
\begin{equation}\label{B1m}
B_1^-  =  - \frac{e^{-k(\phi_0+ \gamma)}}{\pi k\sqrt{\alpha}} \frac{\cosh \pi k \omega}{\sin \pi k \alpha} 
\end{equation}
We then find the $B_0$ coefficients by evaluating at $r=0$,
\begin{equation}\label{B0.def}
B^\pm_0 =  \frac{1}{w_0(0)} \left[ g_k(\tfrac12\pm k\alpha; 0) - B^\pm_1 w_1(0)\right].
\end{equation}
Since $f_k$ and $g_k$ agree at $r=0$, (\ref{fk.w0}) shows that
$$
w_0(0) = A^+_0 g_k(\tfrac12 + k\alpha; 0),
$$
where
$$
A^+_0 :=   \frac{1}{\pi k\sqrt{\alpha}} e^{k(\phi_0 + \gamma)}.
$$

Combining these formulas gives
\begin{equation}\label{gkp.ww}
g_k(\tfrac12 + k\alpha; r) = A^+_0 w_0(r) + B^+_1 \left( w_1(r) - \frac{w_1(0)}{w_0(0)} w_0(r) \right),
\end{equation}
and
\begin{equation}\label{gkm.ww}
g_k(\tfrac12 - k\alpha; r) = \bigl[S_{F_{\ell}}(\tfrac12+k\alpha)\bigr]_k A^+_0 w_0(r)
+  B^-_1 \left( w_1(r) - \frac{w_1(0)}{w_0(0)} w_0(r) \right).
\end{equation}

The asymptotic analysis of (\ref{gkp.ww}) is straightforward.  The $B^+_1 w_1(r)$ term always
dominates for $|k\alpha|$ large and $\arg \alpha \in [0, \tfrac{\pi}2-\vep]$, by Proposition~\ref{wsig.prop}.
By applying Stirling's formula to (\ref{gkp.asym}) we find that
\begin{equation}\label{gkp.est}
g_k(\tfrac12 + k\alpha; r) = \frac{1}{\pi k\sqrt{\alpha}} (\omega^2 + \alpha^2 \cosh^2 r)^{-\frac14}
e^{k(\phi - \phi_0 + \gamma)}(1 + O(|k\alpha|^{-1}).
\end{equation}
The analysis of (\ref{gkm.ww}) more complicated.  This term has both zeros and poles,
and different terms can dominate for $\alpha$ in different regions.  For $\alpha = xe^{i\theta}$,
the borders between these regions will be denoted $x = \varrho_j(\theta)$, $j=1,2$, 
where 
$$
\re\phi_0(\varrho_1(\theta) e^{i\theta}) = 0,\qquad \re \Bigl[\phi(\varrho_2(\theta) e^{i\theta};r) 
- 2 \phi_0(\varrho_2(\theta) e^{i\theta};r)\Bigr] =0.
$$
For the first curve we can be explicit, $\varrho_1(\theta) = \omega \csc \theta$.
\begin{figure} 
\psfrag{iw}{$i\omega$}
\psfrag{iwcosh2r}{$\frac{i\omega}{\cosh^2 r}$}
\psfrag{rp0}{$\re\phi_0 < 0$}
\psfrag{rp1}{$0< \re\phi_0 < \re(\phi-\phi_0)$}
\psfrag{rp2}{$\re\phi_0 > \re(\phi-\phi_0)$}
\psfrag{arho1}{$\alpha = \varrho_1(\theta) e^{i\theta}$}
\psfrag{arho2}{$\alpha = \varrho_2(\theta) e^{i\theta}$}
\begin{center}  
\includegraphics{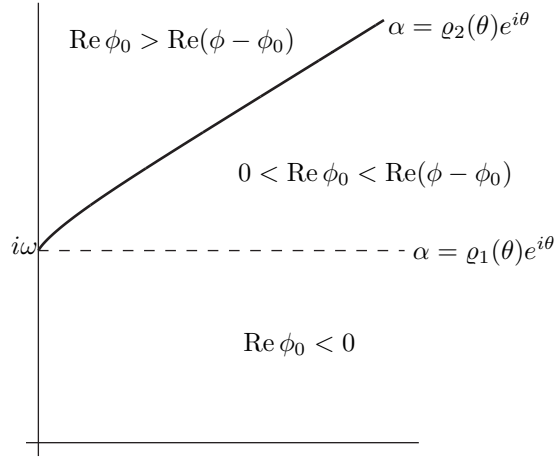} 
\end{center}
\caption{Positive and negative regions for $\re (\phi(\alpha; r) - \phi_0(\alpha))$, shown
for $r = 1$.}\label{phicontour2}
\end{figure}

Consider first the portion of the sum (\ref{exlt.sum}) with $a/k \ge \varrho_2(\theta)$.
In this region, $\re \phi_0 > \re(\phi - \phi_0)$ and the first term
in (\ref{gkm.ww}) dominates the asymptotics.  In this case, provided $|k\alpha| \in\Lambda$,
$$
\log|g_k(\tfrac12 - k\alpha; r)| = k\re(-\phi + \phi_0 - \gamma) + O(\log |k\alpha|).
$$
For $k \le a/\varrho_2(\theta)$, we thus have
$$
 \log \left| \frac{[S_{F_{\ell, -r_0}}(\tfrac12+ae^{i\theta})]_k}{[S_{F_\ell}(\tfrac12+ae^{i\theta})]_k} \right|
= -2k \re \left[ \phi\Bigl(\frac{ae^{i\theta}}{k}; r_0\Bigr) - \phi_0\Bigl(\frac{ae^{i\theta}}{k}\Bigr) \right]
+ O(\log a).
$$
This gives the estimate 
\begin{equation}\label{exsum2}
\begin{split}
& \sum_{1 \le k \le a/\varrho_2(\theta)} \log \left|  
\frac{[S_{F_{\ell, -r_0}}(\tfrac12+ae^{i\theta})]_k}{[S_{F_\ell}(\tfrac12+ae^{i\theta})]_k} \right| \\
&\qquad =   a^2 \int_{\varrho_2(\theta)}^{\infty} \frac{2\re [\phi_0(xe^{i\theta})
-\phi(xe^{i\theta}; r_0)]}{x^3}\>dx
+ O(a\log a).
\end{split}
\end{equation}

The region $\varrho_1(\theta) < a/k < \varrho_2(\theta)$ corresponds to $0 < \re \phi_0 < 
\re (\phi - \phi_0)$.  In this case, 
the $B^-_1w_1(r)$ term dominates the asymptotics of (\ref{gkm.ww}), and we have  
$$
\log |g_k(\tfrac12 - k\alpha; r)| = k\re(\phi - 3\phi_0 - \gamma) + O(\log |k\alpha|).
$$
Using this along with (\ref{gkp.est}) gives 
$$
 \log \left| \frac{[S_{F_{\ell, -r_0}}(\tfrac12+ae^{i\theta})]_k}{[S_{F_\ell}(\tfrac12+ae^{i\theta})]_k} \right|
= -2k \re \phi_0\bigl(ae^{i\theta}/k\bigr) + O(\log a),
$$
for $k \le a/\varrho_2(\theta)$.  We conclude that
\begin{equation}\label{exsum1}
\sum_{a/\varrho_2(\theta) \le k \le a/\varrho_1(\theta)} \log \left|  
\frac{[S_{F_{\ell, -r_0}}(\tfrac12+ae^{i\theta})]_k}{[S_{F_\ell}(\tfrac12+ae^{i\theta})]_k} \right|
= -  a^2 \int_{\varrho_2(\theta)}^{\infty} \frac{2\re \phi_0(xe^{i\theta})}{x^3}\>dx
+ O(a\log a).
\end{equation}

The terms with $\re \phi_0 \le 0$ make only lower order contributions.  First of all,
we can prove a general estimate,
$$
\log \left| \frac{[S_{F_{\ell, -r_0}}(s)]_k}{[S_{F_\ell}(s)]_k} \right| = O((k+|s|) \log |s|),
$$
just as in Lemma~\ref{SFlr.mm}, to show that
\begin{equation}\label{exsum3}
\sum_{\varrho_1(\theta)(1-a^{-1/2}) \le a/k \le \varrho_1(\theta)} \log \left|  
\frac{[S_{F_{\ell, -r_0}}(\tfrac12+ae^{i\theta})]_k}{[S_{F_\ell}(\tfrac12+ae^{i\theta})]_k} \right| 
= O(a^{3/2} \log a).
\end{equation}
For the remaining terms, we use (\ref{gkp.ww}) and (\ref{gkm.ww}) to write
$$
\frac{[S_{F_{\ell, -r}}(\tfrac12+k\alpha)]_k}{[S_{F_{\ell}}(\tfrac12+k\alpha)]_k}
= 1 + \frac{e^{-k(\phi_0+ \gamma)}}{\pi k \sqrt{\alpha}} 
\frac{[S_{F_{\ell}}(\tfrac12-k\alpha)]_k}{g_k(\tfrac12 + k\alpha; r)}
\left( w_1(r) - \frac{w_1(0)}{w_0(0)} w_0(r) \right).
$$
This gives the estimate
$$
\log \left| \frac{[S_{F_{\ell, -r}}(\tfrac12+k\alpha)]_k}{[S_{F_{\ell}}(\tfrac12+k\alpha)]_k}
- 1\right|  \le 2k\re \phi_0(\alpha) + O(\log |k\alpha|).
$$
For $a$ sufficiently large, this gives 
\begin{equation}\label{exsum4}
\sum_{a/k \le \varrho_1(\theta)(1-a^{-1/2})} \log \left|  
\frac{[S_{F_{\ell, -r_0}}(\tfrac12+ae^{i\theta})]_k}{[S_{F_\ell}(\tfrac12+ae^{i\theta})]_k} \right| 
= O(e^{-c\sqrt{a}}).
\end{equation}

The estimates (\ref{exsum1})--(\ref{exsum4}) cover all terms in the sum (\ref{exlt.sum}),
and together yield
\[
\begin{split}
\log |\tau(\tfrac12+ae^{i\theta})| & = 2a^2 \int_{\varrho_2(\theta)}^{\infty} 
\frac{2\re [ 2\phi_0(xe^{i\theta}) - \phi(xe^{i\theta}; r_0)]}{x^3}\>dx \\
&\qquad - \frac{\pi a^2}{\omega} \sin^2 \theta + O(a\log a),
\end{split}
\]
for $a \in \Lambda$ and $0 \le \theta\le  \tfrac{\pi}2 - \vep$.

We now integrate over $\theta \in [0, \tfrac{\pi}2 - \vep]$ and use 
Lemma~\ref{lind.lemma} to control the limit $\vep \to 0$, as in the proof of Theorem~\ref{trfun.thm}. 
This yields
\[
\begin{split}
\frac{2}{\pi} \int_{0}^{\frac{\pi}2} \log |\tau(\tfrac12 + a e^{i\theta})|\>d\theta & = 
\frac{4a^2}{\pi} \int_{0}^{\frac{\pi}2} \int_{\varrho_2(\theta)}^\infty 
\frac{2\re [ 2\phi_0(xe^{i\theta}) - \phi(xe^{i\theta}; r_0)]}{x^3}\>dx\>d\theta \\
&\qquad - \frac{\ell}4 a^2 - o(a^2).
\end{split}
\]
To complete the proof of (\ref{extau.asym}), recall the definition of $\phi(\alpha; r)$
as the integral of $\sqrt{f}\>dr$ in (\ref{zetadiff}).  Since the function $f$ occurring there
is an even function of $r$, $\phi - \phi_0$ will be an odd function of $r$.  (This is not
readily apparent from the definition (\ref{phi.def}).)  This parity implies that
$$
I(\alpha, \ell, -r_0) = 2\re [ 2\phi_0(\alpha) - \phi(\alpha; r_0)]
$$
\end{proof}

\end{document}